\documentclass[12pt,letterpaper]{amsart}
\usepackage{amsthm}
\usepackage{tabularx}
\usepackage{tikz}
\usepackage{parskip}

\usepackage[english]{babel} 
\usepackage[latin1]{inputenc}
\usepackage{amsmath, bm} 
\usepackage{amsfonts}
\usepackage{amssymb}
\usepackage{stmaryrd}
\usepackage{latexsym} 
\usepackage{graphicx}
\usepackage{subfigure}
\usepackage{hyperref}
\usepackage{verbatim}
\usepackage[all]{xy}
\usepackage{graphics}
\usepackage{pdfsync}
\usepackage{xcolor}

\usepackage{listings}
\usepackage[noend]{algpseudocode}
\usepackage{csquotes}
\usepackage{ytableau}
\usepackage{textcmds}
\usepackage{latexsym}

\newtheorem{theorem}{Theorem}[section]
\newtheorem{lemma}[theorem]{Lemma}
\newtheorem{corollary}[theorem]{Corollary}

\newtheorem{prop}[theorem]{Proposition}
\theoremstyle{definition}
\newtheorem{definition}[theorem]{Definition}
\newtheorem{example}[theorem]{Example}

\numberwithin{equation}{section}
\newlength{\cellsize}
\newcommand\tableau[1]{
\vcenter{
\let\\=\cr
\baselineskip=-16000pt
\lineskiplimit=16000pt
\lineskip=0pt
\halign{&\tableaucell{##}\cr#1\crcr}}}

\newcommand{\tableaucell}[1]{{%
\def \arg{#1}\def \void{}%
\ifx \void \arg
\vbox to \cellsize{\vfil \hrule width \cellsize height 0pt}%
\else
\unitlength=\cellsize
\begin{picture}(1,1)
\put(0,0){\makebox(1,1)[c]{$#1$}}
\put(0,0){\line(1,0){1}}
\put(0,1){\line(1,0){1}}
\put(0,0){\line(0,1){1}}
\put(1,0){\line(0,1){1}}
\end{picture}%
\fi}}

\DeclareMathOperator{\sgn}{sgn}

\DeclareMathOperator{\comp}{comp}
\DeclareMathOperator{\rev}{rev}

\DeclareMathOperator{\set}{set}
\newcommand{\dI}{\mathfrak{S}^*}
\newcommand{\rdI}{\mathcal{R}\mathfrak{S}^*}
\newcommand{\rI}{\mathcal{R}\mathfrak{S}}
\newcommand{\Imm}{\mathfrak{S}}

\newcommand{\hdI}{\mathcal{H}\mathfrak{S}^*}

\DeclareMathOperator{\QSym}{QSym}
\DeclareMathOperator{\NSym}{NSym}
\DeclareMathOperator{\Des}{Des}
\DeclareMathOperator{\DesI}{Des_{\dI}}

\DeclareMathOperator{\rw}{rw}
\DeclareMathOperator{\stdz}{stdz}

\newcommand{\bQ}{\mathbb{Q}}

\newcommand{\bZ}{\mathbb{Z}}
\newcommand{\bB}{\mathbb{B}}

\newcommand{\Sym}{\ensuremath{\operatorname{Sym}}}

\newcommand{\Qsym}{\ensuremath{\operatorname{QSym}}}
\newcommand{\qsy}{\hat{\mathcal{S}}}	
\newcommand{\DI}{{\mathfrak{S}}^\ast}	         
\newcommand{\RI}{\mathcal{R}{\mathfrak{S}}^\ast} 
\newcommand{\rqsy}{\mathcal{R}\hat{\mathcal{S}}}	

\newcommand{\Nsym}{\ensuremath{\operatorname{NSym}}}
\newcommand{\nce}{\mathbf{e}}         	
\newcommand{\nch}{\mathbf{h}}         	
\newcommand{\ncr}{\mathbf{r}}           	
\newcommand{\nci}{{\mathfrak{S}}}	         
\newcommand{\ncri}{\mathcal{R}{\mathfrak{S}}} 

\newcommand{\des}{\mathrm{Des}} 

\newcommand{\suchthat}{\;|\;}

\savebox2{%
\begin{picture}(15,15)
\put(0,0){\line(1,0){15}}
\put(0,0){\line(0,1){15}}
\put(15,0){\line(0,1){15}}
\put(0,15){\line(1,0){15}}
\end{picture}}
\newcommand\cellify[1]{\def\thearg{#1}\def\nothing{}%
\ifx\thearg\nothing
\vrule width0pt height\cellsize depth0pt\else
\hbox to 0pt{\usebox2\hss}\fi%
\vbox to 15\unitlength{
\vss
\hbox to 15\unitlength{\hss$#1$\hss}
\vss}}
\newcommand\expath[1]{%
\hbox to 0pt{\usebox3\hss}%
\vbox to 15\unitlength{
\vss
\hbox to 15\unitlength{\hss$#1$\hss}
\vss}}
\newcommand\bas[1]{\omit \vbox to \cellsize{ \vss \hbox to \cellsize{\hss$#1$\hss} \vss}}

\newcommand{\SIT}{\ensuremath{\operatorname{SIT}}}
\newcommand{\iffpsi}{\overset{\psi}{\iff}}
\newcommand{\I}{{\mathfrak{S}}}	 

\newcommand{\sheila}{\textcolor{black}}
\newcommand{\sheilaFeb}{\textcolor{black}} 
\newcommand{\svw}{\textcolor{black}}
\newcommand{\emn}[1]{\textcolor{black}{#1}}
\newcommand{\sheilaMar}{\textcolor{black}} 
\newcommand{\svwMar}{\textcolor{black}} 

\author[Niese, Sundaram, van Willigenburg, Vega, Wang]{Elizabeth Niese, Sheila Sundaram,\\ Stephanie van Willigenburg, Julianne Vega, Shiyun Wang}
\address{Elizabeth Niese: Marshall University, Huntington, WV 25755, USA} 
\email{niese@marshall.edu}
\address{Sheila Sundaram: Pierrepont School, Westport, CT 06880, USA}
\email{shsund@comcast.net}
\address{Stephanie van Willigenburg: University of British Columbia, Vancouver, BC V6T 1Z2, Canada}
\email{steph@math.ubc.ca}
\address{Julianne Vega: Kennesaw State University, Kennesaw, GA 30144, USA}
\email{jvega30@kennesaw.edu}
\address{Shiyun Wang: University of Southern California, Los Angeles, CA 90089-2532, USA}
\email{shiyunwa@usc.edu}
\title{Row-strict dual immaculate functions}
\date{\today}

\begin{document}

\maketitle

\begin{abstract}

We define a new basis of quasisymmetric functions, the row-strict dual immaculate functions, as the generating function of a particular set of tableaux.  We show that this definition gives a function that can also be obtained by applying the  involution $\psi$ to the dual immaculate functions of Berg, Bergeron, Saliola, Serrano, and Zabrocki (2014) \svw{and establish numerous combinatorial properties for our functions.}  We give an equivalent formulation of our functions via Bernstein-like operators, in a similar fashion to Berg et. al (2014).   We conclude the paper by \svw{defining skew dual immaculate functions and hook dual immaculate functions, and establishing combinatorial properties for them.}  
\end{abstract}

\tableofcontents

\section{Introduction}\label{sec:intro}

\svw{Quasisymmetric functions were first defined formally by Gessel \cite{gessel} in relation to the theory of $P$-partitions, and have since grown to be a vibrant area of research in their own right, including playing a crucial role in the resolution of the Shuffle Conjecture \cite{carlsson2018proof}. As a natural nonsymmetric generalization of symmetric functions, one avenue of research has been to establish analogies of classical symmetric functions, for example monomial symmetric functions and chromatic symmetric functions. However an analogy with the ubiquitous Schur functions remained elusive until 2011, when  the authors of \cite{HLMvW2011}  discovered quasisymmetric Schur functions that naturally arose from the combinatorics of nonsymmetric Macdonald polynomials. These functions became the genesis of the now flourishing area of Schur-like functions throughout algebraic combinatorics, for example \svwMar{\cite{ALvW2021, CFLSX2014, CHMMW2022, JL2015, LMvW2013}.} Within the algebra of quasisymmetric functions, two further bases rose to attention: the dual immaculate functions \cite{BBSSZ2014}, and the row-strict quasisymmetric Schur functions \cite{MR2014}, the latter being quasisymmetric Schur functions via the  involution $\psi$. In this paper we will interpolate between these two bases to yield row-strict dual immaculate functions.
}

\svw{More precisely, quasisymmetric} Schur functions, all forms, \svw{can be} defined combinatorially as the generating function of composition fillings (resp. row-strict composition fillings) where there is a requirement that the first column strictly (resp. weakly) increases, each row increases weakly (resp. strictly), and a triple rule is satisfied.  The dual immaculate functions were introduced by Berg et al.~\cite{BBSSZ2014} as the dual basis of the noncommutative symmetric immaculate functions.  Combinatorially the dual immaculate functions can be viewed as the generating functions of composition fillings that satisfy just the first column and row requirements of the quasisymmetric \svwMar{Schur functions}, omitting the triple rule.  

The triple rules required to define all versions of quasisymmetric Schur functions allow those functions to retain \svw{many} of the combinatorial properties of Schur functions, including an RSK-style insertion algorithm, \svw{a JDT algorithm, a  Murnaghan-Nakayama rule, and  Littlewood-Richardson rules.}  Without the triple rule, some combinatorial similarities to Schur functions are lost, but others are gained.  \svw{For example,} the immaculate functions satisfy a noncommutative analogue of the Jacobi-Trudi rule.  

In this paper we define \emph{row-strict immaculate tableaux} of a given composition shape, and study their generating function.  \sheila{By identifying the correct descent set}, we show that our combinatorial definition of the row-strict dual immaculate functions is equivalent to applying the  involution $\psi$ to the dual immaculate functions \svw{in Theorem~\ref{the:RIasDI}}, and can also be obtained \sheila{from the Hopf algebra of noncommutative symmetric functions by suitably defined} creation operators \svw{in Theorem~\ref{thm:row-strict-via-creationops}}.  

We are able to quickly obtain many results from~\cite{BBSSZ2014} by application of the involution $\psi$  \svw{in Theorem~\ref{thm:psiresults}.} \emn{We also carefully construct skew row-strict dual immaculate functions} and define hook dual immaculate functions, \svw{obtaining results for them in our final two sections.}  \sheila{In this work} we focus primarily on combinatorial aspects of the row-strict dual immaculate functions. \sheila{We investigate 0-Hecke modules for these new functions  in~\cite{NSvWVW2022b}.}

\noindent \textbf{Acknowledgments.} The authors would like to thank \sheilaMar{Sarah Mason for bringing to their attention an incorrect equation  in the first version of the paper,}  \svwMar{ the referee for thoughtful comments, and} the Algebraic Combinatorics Research Community program at ICERM through which this research took place. The third author was supported in part by the National Sciences Research Council of Canada.
\section{Background}\label{sec:back} 

In this section we introduce much of the background on quasisymmetric and noncommutative symmetric functions needed for our results.  We refer the reader to~\cite{LMvW2013} for additional details. 

A {\em composition} of a positive integer $n$ is a sequence $\alpha = (\alpha_1,\ldots, \alpha_k)$ such that \svwMar{$\sum _i\alpha_i = n$.} We write $\alpha \vDash n$. \svw{We sometimes denote $n$ by $|\alpha|$ and $k$ by $\ell(\alpha)$.} The diagram of $\alpha = (\alpha_1,\ldots,\alpha_k)$ is a collection of left-justified boxes with $\alpha_i$ boxes in row $i$, where row $1$ is the bottom row.
\begin{example}\label{ex:comps}
For $\alpha = (3,1,4,2,5,1)$, the diagram is \svw{as follows.}
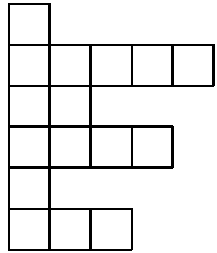
\[\tableau{{}\\{}&{}&{}&{}&{}\\{}&{}\\{}&{}&{}&{}\\{}\\{}&{}&{}}\]
\end{example}

Compositions of $n$ are in bijection with subsets of $\{1,2,\ldots,n-1\}$.  Given a composition $\alpha = (\alpha_1,\svw{\alpha _2 , }\ldots,\alpha_k)$ of $n$, the corresponding set is $\set(\alpha) = \{\alpha_1,\alpha_1+\alpha_2,\ldots,\alpha_1+\cdots+\alpha_{k-1}\}$.  For $\alpha=(3,1,4,2,5,1)$ \svw{that} is a composition of 16, $\set(\alpha) = \{3,4,8,10,15\} \subseteq \{1,2,\ldots,15\}$.
Given a subset $S=\{s_1<\svw{s_2}<\cdots<s_j\}$ of $\{1,\svw{2,} \ldots,n-1\}$, the corresponding composition of $n$ is $\comp(S)=(s_1,s_2-s_1,\ldots,s_j-s_{j-1},n-s_j)$.  For $S=\{2,3,5,9,10,14\} \subseteq\{1,\svw{2,} \ldots,15\}$, $\comp(S)=(2,1,2,4,1,4,2)$. The composition obtained by reversing the order of the parts of $\alpha$, the {\em reverse} of $\alpha$, is $\rev(\alpha) = (\alpha_k, \alpha_{k-1},\ldots, \alpha_1)$. 
The {\em complement} of a composition $\alpha$, denoted $\alpha^c,$ is the composition obtained from $\alpha$ by taking the complement of the set corresponding to  $\alpha$.  That is, 
$\alpha^c= \comp(\set(\alpha)^c)$.
The {\em transpose} of a composition $\alpha$, denoted $\alpha^t,$ is the composition obtained from $\alpha$ by taking the complement of the set corresponding to the reverse of $\alpha$.  That is, 
\[\alpha^t = \comp(\set(\rev(\alpha))^c).\]

For example, if $\alpha = (3,1,2,4)$, $\rev(\alpha) = (4,2,1,3)$, $\set(\rev(\alpha)) = \{4,6,7\}$, $\set(\rev(\alpha))^c = \{1,2,3,5,8,9\}$, so $\alpha^t = (1,1,1,2,3,1,1)$. \svwMar{Note that $\alpha ^ t = \rev(\alpha)^c = \rev(\alpha ^c)$.}

We will use several different orders on compositions. 
\sheila{
For compositions $\alpha$ and  $\beta$, we say $\alpha$ precedes $\beta$ in \emph{lexicographic order}, denoted by $\alpha\le_\ell \beta$, if either $\alpha_1<\beta_1$ or there is a $j>1$ such that $\alpha_j< \beta_j$ but 
$\alpha_i=\beta_i, 1\le i\le j-1$.}
We say that a composition $\beta = (\beta_1, \ldots, \beta_m)$ is a {\em refinement} of a composition $\alpha = (\alpha_1, \ldots, \alpha_k)$, denoted $\beta \preccurlyeq \alpha$, if each part of $\alpha$ can be obtained by adding consecutive parts of $\beta$.  Equivalently, we say that $\alpha$ is a {\em coarsening} of $\beta$.  For example, $\beta = (1,2,1,1,3,2)$ is a refinement of $\alpha = (3,2,5)$.  Finally, we use an order, defined in~\cite{BBSSZ2014}, where $\alpha\subset_s \beta$ if
\begin{enumerate}
    \item $|\beta|=|\alpha|+s$,
    \item $\alpha_j\leq \beta_j ,\,\, \forall~ 1\leq j\leq \svw{\ell(\alpha)}$, and
    \item \svw{$\ell(\beta)\leq \ell(\alpha)+1$.}
\end{enumerate}
\svw{Note that the last two parts guarantee that $\ell(\alpha)\leq \ell (\beta) \leq \ell (\alpha)+1$. If we have only the second condition then this is denoted $\alpha \subseteq \beta$.}

A function $f\in \bQ[[x_1,x_2,\ldots]]$ is {\em quasisymmetric} if the coefficient of $x_1^{\alpha_1}\svw{x_2^{\alpha_2}}\cdots x_k^{\alpha_k}$ is the same as the coefficient of $x_{i_1}^{\alpha_1}x_{i_2}^{\alpha_2}\cdots x_{i_k}^{\alpha_k}$ for every $(\alpha_1,\svw{\alpha _2,}\ldots,\alpha_k)$ and $i_1<i_2<\cdots <i_k$.  The set of all quasisymmetric functions forms a \svw{Hopf algebra} graded by degree, $\QSym = \bigoplus_{n} \QSym_n$, where each $\QSym_n$ is a vector space over $\bQ$ with bases indexed by compositions of $n$. 

The pertinent bases for our purposes include the {\em monomial}, {\em fundamental}, {\em dual immaculate}, and {\em quasisymmetric Schur} bases.  We define the monomial and fundamental bases here and defer the remaining definitions until later.  

Given a composition $\alpha = (\alpha_1,\alpha_2,\ldots,\alpha_k)$ of $n$, the {\em monomial quasisymmetric function} is 
\[M_\alpha = \sum_{\substack{(i_1,i_2,\ldots,i_k)\\\svw{i_1<i_2<\cdots<i_k}}} x_{i_1}^{\alpha_1}x_{i_2}^{\alpha_2}\cdots x_{i_k}^{\alpha_k}.\]
A second important quasisymmetric basis is the fundamental basis.  Given a composition $\alpha=(\alpha_1,\alpha_2,\ldots,\alpha_k)$ of $n$, the {\em fundamental quasisymmetric function} indexed by $\alpha$ is 
\[F_\alpha(x_1,x_2,\ldots) = \sum_{\substack{i_1\leq i_2\leq \cdots \leq i_n\\i_j=i_{j+1} \Rightarrow j\notin\set(\alpha)}} x_{i_1}x_{i_2}\cdots x_{i_n}.\]
 
Note that 
\begin{equation}\label{eq:FtoM}
F_\alpha = \sum_{\beta \preccurlyeq \alpha} M_\beta \quad \text{ and } \quad M_\alpha = \sum_{\svw{\beta \preccurlyeq \alpha}} (-1)^{\ell(\alpha) - \ell(\beta)} F_\beta.
\end{equation}

In~\cite{GKLLRT1995} the {\em noncommutative symmetric functions} are defined as the algebra 
$\NSym=\bQ\langle\nce_1,\nce_2,\ldots\rangle$ generated by noncommuting indeterminates $\nce_n$ of degree $n$.  The set of noncommutative symmetric functions forms a graded \svw{Hopf} algebra $\NSym=\bigoplus_{n} \NSym_n$ where the degree of functions in $\NSym_n$ is $n$.  Each $\NSym_n$ has bases indexed by compositions of $n$.  

The $n$th elementary noncommutative symmetric function is the indeterminate $\nce_n$, where $\nce_0=1$.  Given a composition $\alpha = (\alpha_1,\ldots, \alpha_k)$, we define the {\em elementary noncommutative symmetric function} by 
\[\nce_{\alpha} = \nce_{\alpha_1}\cdots \nce_{\alpha_k}.\]

The $n$th complete homogeneous noncommutative symmetric function is defined by 
\[\nch_n = \sum_{(\alpha_1,\ldots,\alpha_m) \vDash n} (-1)^{n-m} \nce_\alpha\] with $\nch_0=1$.  Then, for $\alpha = (\alpha_1,\ldots, \alpha_k)$, the {\em complete homogeneous noncommutative symmetric function} is defined \svw{by}
\[\nch_\alpha = \nch_{\alpha_1} \cdots \nch_{\alpha_k}.\]

We can write $\nch_\alpha$ in terms of the elementary noncommutative symmetric functions by 
\begin{equation}\label{eq:nchasnce}
\nch_\alpha = \sum_{\beta \preccurlyeq \alpha} (-1)^{|\alpha| - \ell(\beta)} \nce_\beta
\end{equation}
where the sum is over all $\beta$ that refine $\alpha$.  

The {\em noncommutative ribbon Schur function} is defined by 
\begin{equation}\label{eq:ribbon-to-homogeneous-nsym}
\ncr_\alpha = \sum_{\beta \succcurlyeq \alpha} (-1)^{\ell(\alpha) - \ell(\beta)}\nch_\beta
\end{equation}
where the sum is over all $\beta$ that are coarsenings of $\alpha$.

As Hopf algebras, \svwMar{ $\NSym$ and $\QSym$} are dual with the pairing 
\[\langle \nch_\alpha, M_\beta\rangle = \delta_{\alpha\beta}\] and \[\langle \ncr_\alpha, F_\beta\rangle = \delta_{\alpha\beta}\] where $\delta_{\alpha\beta}$ is 1 if $\alpha = \beta$ and 0 otherwise. 

Recall that in $\Sym$ there is an automorphism $\omega: \Sym \rightarrow \Sym$ such that $\omega(s_\lambda) = s_{\lambda'}$ where $\lambda'$ is the transpose of the partition $\lambda$ and $s_\lambda$ denotes the symmetric Schur function.  In $\QSym$ we have three involutive automorphisms~\cite{LMvW2013}, $\psi, \rho,$ and $\omega$ defined on the fundamental basis by 
\begin{align}
\psi(F_\alpha) &= F_{\alpha^c} \label{eqn:psiF}\\
\rho(F_\alpha) &= F_{\rev(\alpha)}; \sheilaMar{\text{ note that  } F_{\rev(\alpha)}(x_1,\ldots,x_n) =F_\alpha(x_n,\ldots, x_1) }\label{eqn:rhoF}\\
\omega(F_\alpha) &= F_{\alpha^t}. \label{eqn:omegaF}
\end{align}
These maps all commute and $\omega = \rho \circ \psi = \psi \circ \rho$.
 
\sheilaMar{
Observe that more generally~\eqref{eqn:rhoF} implies that, for any $f\in\QSym$,
\begin{equation}\label{eqn:rho-any-qsymfn} \rho(f)\,(x_1,\ldots,x_n) =f(x_n,\ldots, x_1).
\end{equation}
}

\sheilaMar{
For completeness, we give a proof of the second statement in~\eqref{eqn:rhoF}, which we were unable to find in the literature. Using the fact  that $j\in \set(\rev(\alpha)) \iff n-j\in \set(\alpha),$ we have 
\begin{align*} F_{\rev(\alpha)}(x_1,\ldots,x_n)
&=\sum_{\substack{1\le i_1\le \cdots\le i_n\le n\\i_j<i_{j+1} \text { if } j\,\in\, \set(\rev(\alpha))}} x_{i_1}\cdots x_{i_n}\\
&=\sum_{\substack{1\le i_1\le \cdots\le i_n\le n\\i_j<i_{j+1} \text { if } n-j\,\in\, \set(\alpha)}} x_{i_1}\cdots x_{i_n}\\
&= \sum_{\substack{n\ge  i_n\ge \cdots\ge i_1\ge 1\\i_{n-j+1}>i_{n-j} \text { if } j\,\in\, \set(\alpha)}} x_{i_n}\cdots x_{i_1}\\&=F_\alpha(x_n, \ldots, x_1).
\end{align*}
Finally the truth of~\eqref{eqn:rho-any-qsymfn}  is evident upon passing to the fundamental expansion of $f\in\QSym$.}

There are corresponding involutions in $\NSym$, denoted by the same letters, and defined on the noncommutative ribbon basis by 
\begin{align}
\psi(\ncr_\alpha) &= \ncr_{\alpha^c}  \quad &\psi(\ncr_\alpha\ncr_\beta) &= \psi(\ncr_\alpha)\psi(\ncr_\beta) \label{eqn:psinsym}\\
\rho(\ncr_\alpha) &= \ncr_{\rev(\alpha)} \quad  &\rho(\ncr_\alpha\ncr_\beta) &= \rho(\ncr_\beta)\rho(\ncr_\alpha) \label{eqn:rhonsym}\\
\omega(\ncr_\alpha)& = \ncr_{\alpha^t} \quad   &\omega(\ncr_\alpha\ncr_\beta) &= \omega(\ncr_\beta)\omega(\ncr_\alpha).\label{eqn:omegansym}
\end{align}
In $\NSym$, $\rho$ and $\omega$ are anti-automorphisms while $\psi$ is an automorphism. 
We also have that $\psi(\nch_\alpha) = \nce_\alpha$, $\rho(\nch_\alpha) = \nch_{\rev(\alpha)}$ and $\omega(\nch_\alpha) = \nce_{\rev(\alpha)}$. 

\begin{prop}\label{prop:invariance-of-pairing-under-psi}
The pairing between \svwMar{$\Nsym$  and $\QSym$} is invariant under the map $\psi$.  That is, for $F\in \Qsym$ and ${\bf g}\in \Nsym$, we have 
\[\langle {\bf g}, F  \rangle=\langle \psi({\bf g}), \psi(F)  \rangle. \]
\begin{proof}
It suffices to check that the equality holds for the noncommutative ribbon basis elements ${\bf g}=\ncr_\alpha$ and the basis of fundamental quasisymmetric functions $F=F_\beta$, where $\alpha, \beta$ are compositions of $n$.  But this is clear from the preceding definitions.
\end{proof}
\end{prop}

Recall from ~\cite[Section 3.4.2]{LMvW2013}, the  \emph{forgetful} map \[\chi:\text{NSym}\longrightarrow \text{Sym}\]
satisfying $\chi(\nce_n)=e_n$, \svwMar{where $e_n$ is the $n$th elementary symmetric function, and similarly we will denote the $n$th complete homogeneous symmetric function by $h_n$.}  For a composition $\alpha\vDash n$, as in \cite[\svw{Section 2.2}]{LMvW2013}, let $\tilde{\alpha}$  be the partition of $n$ obtained by taking the parts of $\alpha$ in \svw{weakly} decreasing order.  Then 
\[\chi(\nch_\alpha)=h_{\tilde{\alpha}}, \quad
\chi(\nce_\alpha)=e_{\tilde{\alpha}}.\]

\begin{prop}\label{prop:chi-psi-Nsym}
 For \svwMar{${\bf g} \in \NSym$, $(\chi \circ \psi)({\bf g})=(\omega \circ \chi)({\bf g}).$}
\end{prop}
\begin{proof} It suffices to verify the equality for the basis \svw{elements} $\nch_\alpha$.  We have 
\[\chi(\psi(\nch_\alpha))=\chi(\nce_\alpha)=e_{\tilde{\alpha}}=\omega(h_{\tilde{\alpha}})=\omega(\chi(\nch_\alpha)),\]
as claimed, \svwMar{where $e_{\tilde{\alpha}}=\omega(h_{\tilde{\alpha}})$ follows by the definition of $\omega$ in $\Sym$.}
\end{proof}

\subsection{Dual immaculate functions}\label{subsec:dualimm}

The immaculate functions $\mathfrak{S}_\alpha$ are a basis of $\NSym$  formed by iterated creation operators~\cite{BBSSZ2014}.  Their \svw{duals} in $\QSym$ \svw{form the basis consisting of} \svw{dual immaculate functions}, $\dI_\alpha$.  These functions can be defined combinatorially as the generating function for \svw{ immaculate tableaux}.  
\begin{definition}
Given a composition $\alpha$, an \emph{immaculate tableau} \svwMar{of \emph{shape} $\alpha$} is a filling, $D$, of the cells of the diagram of $\alpha$ with positive integers such that 
\begin{enumerate} 
\item The leftmost column entries strictly increase from bottom to top.
\item The row entries weakly increase from left to right.
\end{enumerate}
\end{definition}
An   immaculate tableau  of shape $\alpha\vDash n$ is \emph{standard} if it is filled with distinct entries taken from $\{1,2,\ldots,n\}.$  
Given an immaculate tableau $D$, we form a {\em content monomial}, $x^D$, by setting the exponent of $x_i$ to be $d_i$, the number of $i$'s in the tableau $D$, \svw{namely,} $x^D = x_1^{d_1}x_2^{d_2}\cdots x_{k}^{d_k}$. \sheila{We call the vector $(d_1,d_2,\ldots)$ the \emph{content} of the tableau $D$.  In particular a standard tableau of shape $\alpha\vDash n$ has content equal to the composition  $(1^n)$.}
\begin{definition} \label{def:dIfunction}
The {\em dual immaculate function} indexed by the composition $\alpha$ is \[\dI_\alpha = \sum_D x^D\] where the sum is over all immaculate tableaux of shape $\alpha$.
\end{definition}

We can rewrite the dual immaculate functions in terms of the fundamental basis as a sum over \svw{standard} immaculate tableaux.  
To do this, we first standardize each immaculate tableau and define a descent set on the standard immaculate tableaux.  
The {\em reading word} of an immaculate tableau $D$ is obtained by reading the entries of $D$ from left to right starting with the top row.  We can standardize a semi-standard tableau (repeated entries allowed) by replacing all the 1's in the reading word by 1,2,\ldots, in reading order, then the 2's, etc.

\begin{example}\label{ex:imtab}
Here is an immaculate tableau of shape $\alpha = (3,2,4,1,2)$ that has reading word 6\,7\,5\,3\,4\,4\,5\,2\,2\,1\,1\,2, and its standardization.
\[T=\tableau{6&7\\5\\3&4&4&5\\2&2\\1&1&2} \qquad S=\tableau{11&12\\9\\6&7&8&10\\3&4\\1&2&5}\]
\end{example}

\sheila{For a composition $\alpha$, let $\SIT(\alpha)$ denote the set of standard immaculate tableaux of  shape $\alpha$.}

\begin{definition}\label{def:dessetdI}\cite[Definition~3.20]{BBSSZ2014}
Given a standard immaculate tableau $S$, the {\em descent set} of $S$, denoted $\DesI(S)$, is 
\[\DesI(S)=\{i: i+1 \text{ appears strictly above }i \text{ in } S \}.\] 
\sheila{We refer to $\DesI(S)$ as the \emph{$\dI$-descent set} of $S$, and to its associated composition $\comp(\DesI(S))$ as the \emph{$\dI$-descent composition} of $S$.}
\end{definition} For the standard immaculate tableau in Example~\ref{ex:imtab}, $\DesI(S) = \{2,5,8,10\}$.

Then \cite[Proposition~3.37]{BBSSZ2014}
\begin{equation}\label{eqn:dI-fund}\dI_\alpha = \sum_{S\in\SIT(\alpha)} F_{\comp(\DesI(S))},\end{equation} 
where the sum is over all standard immaculate tableaux of shape $\alpha$.

\section{Row-strict dual immaculate functions}\label{sec:RSDI}
In this section we start with a combinatorial definition of a new quasisymmetric \svw{function} that we call the \svw{ row-strict dual immaculate function.}

\begin{definition} Given a composition $\alpha$, a \emph{row-strict immaculate tableau} \svwMar{of \emph{shape} $\alpha$} is a filling, $U$, of the \svwMar{cells of the} diagram of $\alpha$ with positive integers such that 
\begin{enumerate}
\item The leftmost column entries weakly increase from bottom to top.
\item The row entries strictly increase from left to right.
\end{enumerate}

\svwMar{We now define our new function, where $x^U$ is the content monomial of the tableau $U$, defined just before Definition~\ref{def:dIfunction}. The content of the tableau $U$ is also defined just before Definition~\ref{def:dIfunction}.}

\svwMar{\begin{definition}\label{def:rdI function}
    The {\em row-strict dual immaculate function} indexed by $\alpha$ is $$ \rdI_\alpha = \sum_U x^U$$where the sum is over all row-strict immaculate tableaux of {shape} $\alpha$. 
\end{definition}}

We say the row-strict tableau $U$ is standard if $x^U=x_1\cdots x_n$. Thus standard row-strict immaculate tableaux coincide with standard immaculate tableaux.
\end{definition}
As before, standardization provides us with a way to expand $\rdI_\alpha$ in terms of the fundamental basis using only standard  tableaux.  

\begin{definition}\label{def:rsreadingword}
Given a row-strict immaculate tableau $T$, the \emph{row-strict immaculate reading word} of $T$, denoted $\rw_{\rdI}(T)$, is the word obtained by reading the entries in the rows of $T$ from right to left starting with the bottom row and moving up.
\end{definition}

To {\em standardize} a row-strict immaculate tableau $T$, replace the 1's in $T$ with \svw{$1,2,\ldots,$} in the order they appear in $\rw_{\rdI}(T)$, then the 2's, etc.  

\begin{definition}\label{def:rsdesset}
The \svw{\emph{descent set}} of a standard row-strict immaculate tableau \svw{$T$} is the set 
\[\Des_{\rdI}(T)=\{i:i+1\text{ is weakly below } i \text{ in } \svw{T}\}.\]
\sheila{We refer to $\Des_{\rdI}(T)$ as the \svwMar{\emph{$\rdI$-descent set}} of $T$, and to its associated composition $\comp(\Des_{\rdI}(T))$ as the \svwMar{\emph{$\rdI$-descent composition} of $T$.}}
\end{definition}
\begin{example}
Consider the row-strict immaculate tableau 
\[T=\tableau{4\\3&4&5&6\\2&5\\1&2&6}\]
The row-strict immaculate reading word of $T$ is $6\,2\,1\,5\,2\,6\,5\,4\,3\,4$ and the corresponding standardized row-strict immaculate tableau is 
\[S=\tableau{6\\4&5&8&10\\3&7\\1&2&9}\] 

Here $\Des_{\rdI}(T) = \{1,4,6,8\}$.

\end{example}

The row-strict dual immaculate functions expand positively in the fundamental basis. 
\begin{theorem}\label{thm:rsfunddecomp}
Let $\alpha \vDash n$.  Then 
\[\rdI_\alpha = \sum_{S} F_{\comp(\Des_{\rdI}(S))}\]
where the sum is over all standard row-strict immaculate tableaux of shape $\alpha$.
\end{theorem}
\begin{proof}
Let $T$ be a row-strict immaculate tableau of shape $\alpha$. Then $T$ standardizes to some standard row-strict immaculate tableau $S$.  Suppose $i\in \Des_{\rdI}(S)$.  Then $i+1$ is weakly below $i$ in $S$.  If $i$ and $i+1$ are in the same row of $S$, then the entry of $T$ replaced by $i$ is strictly less than the \svwMar{entry} replaced by $i+1$ since rows of $T$ strictly increase.  If $i+1$ is in a lower row than $i$, then the entry of $T$ replaced by $i$ must be strictly less than the entry replaced by $i+1$, else the standardization process was not followed.  Thus $x^T$ has strict increases at each position in $\Des_{\rdI}(S)$ and $x^T$ is a monomial in $F_{\comp(\Des_{\rdI}(S))}$.  Thus every monomial in $\rdI_\alpha$ appears in $\sum_S F_{\comp(\Des_{\rdI}(S))}$.

Now let $S$ be a standard row-strict immaculate tableau and let $x_{i_1}\cdots x_{i_n}$ with $i_1\leq i_2\leq \cdots \leq i_n$ be a monomial in $F_{\comp(\Des_{\rdI}(S))}$.  Create a new diagram $T$ from $S$ by replacing each entry $k$ in $S$ with $i_k$.  If $i_k=i_{k+1}$ then $k \notin \Des_{\rdI}(S)$, so $k$ must appear strictly below $k+1$ in $S$ and thus each entry in a row of $T$ is distinct and increases left to right.  By construction, the first column will weakly increase from bottom to top.  Thus $T$ is a semi-standard row-strict immaculate tableau with content $(i_1,\svw{\ldots,} i_n)$, and $x_{i_1}\cdots x_{i_n}$ is a monomial in $\rdI_\alpha$. 
\end{proof}

\begin{example}
Let 
\[S=\tableau{6\\4&5&8&10\\3&7\\1&2&9}\] be a standard row-strict immaculate tableau.  Then $\Des_{\rdI}(S)=\{1,4,6,8\}$ and $x^P=x_1x_2^2x_3x_4^2x_5^2x_6^2$ is a monomial in $F_{\comp(\Des_{\rdI}(S))}$.  We can ``destandardize'' $S$ as described in the proof of Theorem~\ref{thm:rsfunddecomp} to obtain 
\[T=\tableau{4\\3&4&5&6\\2&5\\1&2&6}.\]
\end{example}

For any standard immaculate tableau $S$, \svw{note by definition that} $\des_{\DI}(S) = \des_{\rdI}(S)^c$.

It will be helpful to know how the involutions $\psi, \rho,$ and $\omega$ act on $\dI_\alpha$.

\begin{theorem}\label{the:RIasDI}
Let $\alpha$ be a composition. Then
\begin{align}
\psi(\dI_\alpha)& =\rdI_\alpha \label{eqn:psidualimm}\\
\rho(\dI_\alpha(x_1, \ldots, x_n)) &=\dI_{\alpha}(x_n, \ldots, x_1)
\label{eqn:rhodualimm}\\
\omega (\DI _\alpha (x_1, \ldots, x_n))
&= \RI _\alpha \svw{(x_n, \ldots , x_1) }\label{eqn:omegadualimm}
\end{align}
\end{theorem}

\begin{proof} Let $\alpha$ be a composition.  Recall from \eqref{eqn:psiF} that $\psi(F_\alpha)=F_{\alpha^c}$. 
Then 
\begin{align*}
\psi(\DI_\alpha) & = \psi\left(\sum_S F_{\comp(\des_{\DI}(S))}\right)\\
&=\sum_S \psi(F_{\comp(\des_{\DI}(S))})\\
&=\sum_S F_{\comp(\des_{\DI}(S)^c)}\\
&=\sum_S F_{\comp(\des_{\rdI}(S))}\\
&=\rdI_\alpha .
\end{align*}
\sheilaMar{The second equation follows from~\eqref{eqn:rho-any-qsymfn}.   Finally the third equation is now a consequence of the fact that $\omega=\rho\circ\psi$.}
\end{proof}

\begin{corollary}\label{cor:basis}
We have that $\{ \RI _\alpha \suchthat \alpha \vDash n\}$ is a basis for $\Qsym _n$.
\end{corollary}

\begin{proof} Since $\{ \DI _\alpha \suchthat \alpha \vDash n\}$ is a basis for $\Qsym _n$ and $\psi$ is an involution it follows by Theorem~\ref{the:RIasDI} that $\{ \RI _\alpha \suchthat \alpha \vDash n \}$ is also a basis for $\Qsym _n$.
\end{proof}

Recall from Section~\ref{subsec:dualimm}  that the \emph{immaculate functions} $\nci _\beta$ form a basis of  $\NSym$ satisfying, by definition,
$$\langle \nci_\alpha,\DI_\beta \rangle = \delta _{\alpha\beta}.$$

Similarly, by definition, we have \emph{row-strict immaculate functions} $\ncri _\beta$ 
 in $\NSym$ satisfying
$$\langle \ncri_\alpha,\RI_\beta \rangle = \delta _{\alpha\beta}.$$
An immediate consequence of these definitions is the effect of the map $\psi$ on $\Imm_\alpha$.  
 Using Proposition~\ref{prop:invariance-of-pairing-under-psi}, we have,  by duality, 
\[\delta_{\alpha\beta} = \langle \Imm_\alpha,\dI_\beta\rangle = \langle \psi(\Imm_\alpha),\psi(\dI_\beta)\rangle = \langle \psi(\Imm_\alpha), \rdI_\beta\rangle,\]
and hence $\psi(\Imm_\alpha) = \rI_\alpha$.

From~\cite[Proposition 3.36]{BBSSZ2014} we have that the dual immaculate functions are monomial positive:
\[\dI_\alpha=\sum_{\beta\le_\ell\, \alpha} K_{\alpha, \beta} M_\beta,\]
\sheilaMar{where $K_{\alpha, \beta}$ is the number of  immaculate tableaux of shape $\alpha$ and content $\beta$}.  It follows from this expansion that 
 $K_{\alpha, \beta}=\langle \nch_\beta, \dI_\alpha\rangle=\langle \nce_\beta, \rdI_\alpha\rangle.$
Similarly, for \svw{row-strict dual immaculate functions,}  we have \svw{by their definition and that of monomial quasisymmetric functions that}
\[\rdI_\alpha=\sum_{\beta} K^*_{\alpha, \beta} M_\beta\]

where $K^*_{\alpha, \beta}$ is the number of row-strict immaculate tableaux of shape $\alpha$ and  content $\beta$, and 
$K^*_{\alpha, \beta}=\langle \nch_\beta, \rdI_\alpha\rangle=\langle \nce_\beta, \dI_\alpha\rangle.$ 
Note that $K_{\alpha,\beta}\neq K^*_{\alpha,\beta}$ in general, and  $K^*_{\alpha,\alpha}=0$ unless $\ell(\alpha)\ge \alpha_1.$ %

\sheilaMar{ 
By contrast, it is easy to see  \cite[Proposition 3.36]{BBSSZ2014} that the transition matrix $(K_{\alpha, \beta})$ is upper unitriangular:
\begin{equation}\label{eqn:dualImm-Kostka}
K_{\alpha, \alpha}=1 \text{ and } K_{\alpha, \beta}\ne 0 \Longrightarrow 
\beta\le_\ell\, \alpha. \end{equation} 
\begin{example} Let $\alpha=(1,2)$. \svwMar{Then the} tableau definitions give the monomial expansions 
\[ \rdI_{12}=M_{21} +M_{111},\quad\dI_{12}= M_{12}+M_{111}.  \]
\end{example}
 As this example shows,  the triangularity condition of \eqref{eqn:dualImm-Kostka} is false for the transition matrix $K^*_{\alpha,\beta}$, making it difficult to deduce, directly from the monomial expansion, that the row-strict dual immaculate functions form a basis.
}

Let $L_{\alpha,\beta}$ denote the number of standard immaculate tableaux of shape $\alpha$ with $\dI$-descent composition $\beta$ and $L^*_{\alpha,\beta}$ denote the number of standard immaculate tableaux of shape $\alpha$ with $\rdI$-descent composition $\beta$.  Given a standard immaculate tableau $T$, we have $\Des_{\dI}(T)^c=\Des_{\rdI}(T)$, and so \svw{$L^*_{\alpha,\beta} = L_{\alpha,\beta^c}$}. 

\begin{theorem}\label{the:KasL}
Fix a composition $\alpha$. For any composition $\gamma$ with $|\gamma|=|\alpha|$, 
\[K^*_{\alpha, \gamma}=\sum_{\beta \succcurlyeq \gamma}L_{\alpha,\beta^c}=\sum_{\beta \succcurlyeq \gamma} \svw{L^*_{\alpha,\beta}.}\]
\end{theorem}

\begin{proof}
We have 
\[\rdI_\alpha=\sum_\gamma K^*_{\alpha, \gamma} M_\gamma,\]
and 
\[\rdI_\alpha=\sum_{T\in\SIT(\alpha)} F_{\comp(\Des_{\rdI}(T))} 
=\sum_\beta  \svw{L^*_{\alpha,\beta}} \, F_\beta=\sum_\beta  L_{\alpha,\beta^c}\, F_\beta.\]
Since the monomial expansion of $F_\beta$ is $F_\beta=\sum_{\gamma\preccurlyeq\beta} M_\gamma$, 
equating coefficients of $M_\gamma$ gives 
\[K^*_{\alpha, \gamma}=\sum_{\beta \succcurlyeq \gamma}L_{\alpha,\beta^c}=\sum_{\beta\succcurlyeq\gamma} \svw{L^*_{\alpha,\beta}.}\qedhere\]
\end{proof}

\subsection{Creation operators and row-strict immaculate functions}\label{subsec:ops}

In \cite{BBSSZ2014}, the authors  defined  a family of operators on $\Nsym$, modelled after Bernstein's operators \svw{that} were used to define the ordinary Schur functions in the Hopf algebra of symmetric functions \svw{\cite[pp.~95-97, Exercise 29]{MACD1995}.}   This new family of ``creation operators" was then used to define the immaculate basis of $\Nsym$, and, via the pairing between $\Nsym$ and its dual $\Qsym$, the dual immaculate quasisymmetric \svw{functions}  $\dI_\alpha$.

In this section we define a variant of the creation operators of \cite{BBSSZ2014}, and show how they in turn lead to a definition of the row-strict immaculate basis \svw{of $\NSym$} and our row-strict  \svw{dual} immaculate quasisymmetric functions $\rdI_\alpha$. %

A pair of  dual Hopf algebras $A$ and $B$ over a field $\mathbb{K}$ induces a pairing $\langle, \rangle:A\times B\rightarrow \mathbb{K}$.  Hence for each element $F\in B$, one can define the adjoint operator $F^{\perp}:A\rightarrow A$ by 
\[\langle F^\perp(a), b\rangle \svw{ = } \langle a, Fb\rangle.\] 
Explicitly, if \svw{$\{a_\alpha\}$} and \svw{$\{b_\alpha\}$} are bases of $A$ and $B$ respectively so that \svw{$\langle a_\alpha, b_\beta\rangle=\delta_{\alpha\beta}$ as before,}  then the operator $F^{\perp}$ may be computed according to the formula
\begin{equation}\label{eqn:adjoint}
F^{\perp}( g) \svw{ = } \sum_{\svw{\alpha}} \langle g,F \svw{b_\alpha}\rangle \svw{a_\alpha}.
\end{equation}

As in \cite{BBSSZ2014}, we apply this to the graded dual Hopf algebras $A=\Nsym$ and $B=\Qsym$.  Let $\{F_\alpha\}_{\alpha\vDash n}$ be the basis of fundamental quasisymmetric functions in $\Qsym$, indexed by the compositions $\alpha$ of the nonnegative integer $n$.  We will consider the linear transformation $F_\alpha^\perp$ of $\Nsym$ \svw{that} is adjoint to multiplication by $F_\alpha$ in $\Qsym$.  

First we record the following important effect of the involution $\psi$ on the adjoint transformation.
\begin{prop}\label{prop:psi-adjoint} Let $F\in \Qsym, H\in \Nsym$.  Then 
\[\psi[F^\perp(\psi(H))]=[\psi(F)]^\perp(H),\] or equivalently, \[ \psi[F^\perp(H)]=[\psi(F)]^\perp(\psi(H)).\]
In particular, for the fundamental quasisymmetric function $F_\alpha$ indexed by the composition $\alpha$, we have 
$F^\perp_\alpha(\psi(H))=\psi[F^\perp_{\alpha^c}(H)]$ 
and hence 
\[F^\perp_{\svw{(1^i)}}(\psi(H))=\psi[F^\perp_{(i)}(H)], \quad F^\perp_{(i)}(\psi(H))=\psi[F^\perp_{(1^i)}(H)].\]
\end{prop}
\begin{proof} Let $\{a_\alpha\}_{\alpha\vDash n}$ and $\{b_\alpha\}_{\alpha\vDash n}$ be dual  bases of  $\Nsym$ and $\Qsym$ respectively, so that $\langle a_\alpha, b_\beta\rangle=\svw{\delta_{\alpha \beta}}$. 

From Equation~\eqref{eqn:adjoint} we have 
\begin{equation*} F^\perp(\psi(H))=\sum_\alpha \langle \psi(H), F b_\alpha\rangle a_\alpha
=\sum_\alpha \langle H, \psi(F)\psi( b_\alpha)\rangle a_\alpha
\end{equation*}
by Proposition~\ref{prop:invariance-of-pairing-under-psi}, and hence 
\begin{equation*} \psi[F^\perp(\psi(H))]
=\sum_\alpha \langle H, \psi(F)\psi( b_\alpha)\rangle \psi(a_\alpha)
=[\psi(F)]^\perp(H),
\end{equation*}
since again Proposition~\ref{prop:invariance-of-pairing-under-psi} implies that  duality of bases is preserved under $\psi$. 
\end{proof}

\begin{lemma}\label{lem:BBSSZ2014-lem2.6}\cite[Lemma 2.6]{BBSSZ2014} 
For $i,j>0$ and \svw{$f\in \NSym$,} 
\[F^\perp_{\svw{(1^i)}}(f\nch_j)=F^\perp_{\svw{(1^i)}}(f) \nch_j+ F^\perp_{\svw{(1^{i-1})}}(f)\nch_{j-1};\quad
F^\perp_{\svw{(i)}}(f\nch_j)=\sum_{k=0}^{\min(i,j)} F^\perp_{\svw{(i-k)}}(f) \nch_{j-k}.\]
In particular we have 

$ F_{\svw{(i)}}^\perp(\nch_j)=\begin{cases} 0, & i>j\\
                                            \nch_{j-i}, & 1\le i\le j\\
                                             \nch_j, &i=0;
\end{cases}$
\quad
$ F_{\svw{(1^i)}}^\perp(\nch_j)=\begin{cases} 0, & i>1\\ 
                                            \nch_{j-1}, & i=1\\
                                             \nch_j, &i=0.
\end{cases}$
\end{lemma}

The next two definitions are made in \cite{BBSSZ2014}.
\begin{definition}\label{def:BBSSZ-creation1}\cite[Definition 3.1]{BBSSZ2014} The \svwMar{\emph{noncommutative Bernstein operator}} $\bB_m$ is defined by 
$$\bB_m 
\svw{=} \sum_{i\ge 0} (-1)^i \nch_{m+i} F^\perp_{\svw{(1^i)}},$$
and for  $\alpha\in \bZ^m$, 
\[ \bB_\alpha \svw{=}\bB_{\alpha_1}\cdots\bB_{\alpha_m}.\]
Note that when $i=0$, \svw{$(1^0)$} is the empty composition and thus $F^\perp_{\svw{(1^0)}}(f)=f=F^\perp_\emptyset(f)$ for all $f\in\Nsym$, since $F_\emptyset=1$ in $\Qsym$.  
Also $F^\perp_{\svw{(1^i)}}(1)=F^\perp_{\svw{(i)}}(1)=\begin{cases} 0 & i>0,\\ 1 &i=0. \end{cases}$ 
\end{definition}

\svw{While we chose duality to define immaculate functions, the following is the original definition, which was proven to be equivalent in \cite{BBSSZ2014}.}
\begin{definition}\label{def:BBSSZ-creation2}\cite[Definition 3.2]{BBSSZ2014} For  any $\alpha\in \bZ^m$, \svwMar{the \emph{immaculate function} $\nci_\alpha\in \Nsym$ is given} by 
\[ \nci_\alpha \svw{=} \bB_\alpha(1)=\bB_{\alpha_1}\cdots\bB_{\alpha_m}(1).\]
\end{definition}

This definition was inspired by Bernstein's original definition in the Hopf algebra of symmetric functions for a Schur function $s_\alpha$ indexed by any $m$-tuple $\alpha\in \bZ^m$.

As observed in \cite[Example 3.3]{BBSSZ2014}, we have
\[ \nci_{(m)}=\bB_{m}(1)=\nch_m, \ \nci_{(a,b)}=\bB_a(\nch_b)=\nch_a\nch_b-\nch_{a+1}\nch_{b-1}.\]
Applying $\psi$ to Lemma~\ref{lem:BBSSZ2014-lem2.6}, and using Proposition~\ref{prop:psi-adjoint} and the fact that $\psi(F_\alpha)=\svw{F_{\alpha^c}}$, so that \svw{$\psi(F_{(1^i)})=F_{(i)}$} in $\Nsym_i$, we obtain
\begin{lemma}\label{lem:rs-adjoint-lem2.6} For $i,j>0$ and $f\in \Nsym$, 
\[F^\perp_{\svw{(i)}}(f\nce_j)=F^\perp_{\svw{(i)}}(f) \nce_j+ F^\perp_{\svw{(i-1)}}(f)\nce_{j-1};\qquad
F^\perp_{\svw{(1^i)}}(f\nce_j)=\sum_{k=0}^{\min(i,j)} F^\perp_{\svw{(1^{i-k})}}(f) \nce_{j-k}.\]
In particular we have 

$ F_{\svw{(1^i)}}^\perp(\nce_j)=\begin{cases} 0, & i>j\\
                                            \nce_{j-i}, & 1\le i\le j\\
                                             \nce_j, &i=0;
\end{cases}$
\quad
$ F_{\svw{(i)}}^\perp(\nce_j)=\begin{cases} 0, & i>1\\ 
                                            \nce_{j-1}, & i=1\\
                                             \nce_j, &i=0.
\end{cases}$
\end{lemma}

Now we define new  operators as follows.
\begin{definition}\label{def:psi-creation} Define the \svwMar{\emph{noncommutative Bernstein operator}} $\bB^{r\!s}_m$  by 
$$\bB^{r\!s}_m \svw{=} \sum_{i\ge 0} (-1)^i \nce_{m+i} F^\perp_{\svw{(i)}},$$
and for  $\alpha\in \bZ^m$, 
\[ \bB^{r\!s}_\alpha \svw{=} \bB^{r\!s}_{\alpha_1}\cdots \bB^{r\!s}_{\alpha_m}.\]

Note that when $i=0$, \svw{this is the empty composition and $F_\emptyset=1$ in $\Qsym$, and thus $F^\perp_{(0)}(f)=f = F^\perp_{\emptyset}(f)$} for all $f\in\Nsym$. %
\end{definition}

Furthermore we have \svw{the following.}
\begin{lemma} \label{lem:row-strict-creation}  For  $\alpha\in \bZ^m$, 
$\psi(\nci_\alpha)=\bB^{r\!s}_\alpha(1)$.
\end{lemma}
\begin{proof} From the above properties, it is clear that 
\[ \bB^{r\!s}_{m}(1)=\nce_m, \ \psi(\nci_{(a,b)})=\bB^{r\!s}_a(\nce_b)=\nce_a\nce_b-\nce_{a+1}\nce_{b-1}.\]
Hence the result is true for $m\le 2$.   Let $f\in \Nsym$. We claim that 
\begin{equation}\label{eqn:key-psi-op}\psi(\bB_{m}(f)) = \bB^{r\!s}_{m}(\psi(f)).\end{equation}
We have 
\begin{equation*}\begin{gathered}\psi(\bB_{m}(f))=\psi \left[\sum_{i\ge 0} (-1)^i \nch_{m+i} F_{\svw{(1^i)}}^\perp(f)\right]
=\sum_{i\ge 0} (-1)^i \nce_{m+i} \psi [ F_{\svw{(1^i)}}^\perp(f)]\\
=\sum_{i\ge 0} (-1)^i\nce_{m+i} F_{\svw{(i)}}^\perp(\psi(f))=\bB^{r\!s}_{m}(\psi(f))\svw{,}
\end{gathered}
\end{equation*}
where  the penultimate equality is thanks to Proposition~\ref{prop:psi-adjoint}.

Since  for  $\alpha\in \bZ^m$, 
\[\bB_\alpha(1)=\bB_{\alpha_1}(f), \ f=\bB_{\alpha_2}\cdots\bB_{\alpha_m}(1),\]
the result now follows by induction.\end{proof}

\begin{theorem}\label{thm:row-strict-via-creationops} The row-strict immaculate function $\ncri_\alpha$ can be defined as the result of applying a creation operator as follows:
\[\ncri_\alpha=\bB^{r\!s}_\alpha(1).\]
\end{theorem}
\begin{proof} Immediate from the preceding lemma, since we already know that $\ncri_\alpha=\psi(\nci_\alpha)$.
\end{proof}

Finally, just as left multiplication by $\nch_m$ can be expressed in terms of creation operators 
\cite[Remark 3.6]{BBSSZ2014}, we have \svw{the following.}
\begin{lemma}  Left multiplication by $\nch_m$ in $\Nsym$ can be expressed as applying the operator
\[\nch_m=\sum_{i\ge 0} \bB_{m+1} F_{\svw{(i)}}^\perp,\]
and left multiplication by $\nce_m$ in $\Nsym$ can be expressed as  applying the operator
\[\nce_m=\sum_{i\ge 0} \bB^{r\!s}_{m+1} F_{\svw{(1^i)}}^\perp.\]
\end{lemma}
\begin{proof} Immediate from Equation~\eqref{eqn:key-psi-op}.
\end{proof}

\subsection{Results obtained by using $\psi$}\label{sec:psiresults}

We can immediately obtain the row-strict analogue of many results in~\cite{BBSSZ2014} by using the involution  $\psi$.  We list here the most pertinent for the remainder of the paper.  We leave results for skew row-strict dual immaculate functions to the next section, as the combinatorial definition is not obviously equivalent.  

\begin{theorem}\label{thm:psiresults}
\begin{enumerate}
\item \cite[Lemma 3.4]{BBSSZ2014} For $s\ge 0, m\in \bZ$ and $f\in \Nsym$, 
\begin{align*}
\bB_m(f)\nch_s&=\bB_{m+1}(f) \nch_{s-1}+\bB_m(f\nch_s)\\
\iffpsi
\bB^{r\!s}_m(f)\nce_s&=\bB^{r\!s}_{m+1}(f) \nce_{s-1}+\bB^{r\!s}_m(f\nce_s).
\end{align*}

\item \label{thm:pieri} \cite[Theorem 3.5]{BBSSZ2014} (Multiplicity-free right Pieri rule) For a composition $\alpha$ and $s \geq 0$,
\[\nci_\alpha\nch_s=\sum_{\alpha_{\subset_s}\beta} \nci_\beta 
\iffpsi
\ncri_\alpha\nce_s=\sum_{\alpha_{\subset_s}\beta} \svw{\ncri_\beta .} \]

\item \label{thm:pieri2} \cite[Proposition 3.32]{BBSSZ2014} (Another multiplicity-free right Pieri rule)
For a composition $\alpha$ and $s \geq 0$, 
\[\nci_\alpha \nci_{\svw{(1^s)}} =\nci_\alpha \nce_{s} =\sum\limits_{\beta} \nci_\beta \iffpsi
\ncri_\alpha \ncri_{\svw{(1^s)}} =\ncri_\alpha \nch_s= \sum\limits_{\beta} \ncri_\beta,\]
where the summation ranges over compositions of $\beta$ of $|\alpha| + s$ such that $\alpha_i \leq \beta_i \leq \alpha_i+1$ and $\alpha_i = 0$ for $i > \ell(\alpha)$. 

\item \cite[Corollary 3.31]{BBSSZ2014} \svwMar{For $n\geq 0$,}
\[\nci_{\svw{(1^n)}}=\sum_{\alpha\vDash n} (-1)^{n-\ell(\alpha)} \nch_\alpha = \nce_n
\iffpsi
\ncri_{\svw{(1^n)}}=\sum_{\alpha\vDash n} (-1)^{n-\ell(\alpha)} \nce_\alpha = \nch_n.\]

\item \cite[Theorem 3.27]{BBSSZ2014} (Jacobi-Trudi) 
For $\ell(\alpha) = m$,  
\[\nci_\alpha = \sum_{\sigma \in S_m}(-1)^{\svw{\sgn(\sigma)}} \nch_{(\alpha_1+\sigma_1-1,\alpha_2+\sigma_2-2,\ldots,\alpha_m+\sigma_m-m)}\]
$\iffpsi$
\[\ncri_\alpha = \sum_{\sigma \in S_m}(-1)^{\svw{\sgn(\sigma)}} \nce_{(\alpha_1+\sigma_1-1,\alpha_2+\sigma_2-2,\ldots,\alpha_m+\sigma_m-m)}\]\svw{where $S_m$ is the symmetric group on $m$ elements and $(-1)^{\sgn(\sigma)}$ is the sign of $\sigma$.}

\item %
From \cite[\svw{Lemma 2.5}]{BBSSZ2014} and Equation~\eqref{eqn:key-psi-op},
\[ F^\perp_{\svw{(1^r)}}(\nci_{\svw{(1^n)}})=\nci_{\svw{(1^{n-r})}}, \text{ and for } s>1, F^\perp_{\svw{(s)}}(\nci_{\svw{(1^n)}})=0\]
which is equivalent via the map $\psi$ to 
\[ F^\perp_{\svw{(r)}}(\ncri_{\svw{(1^n)}})=\ncri_{\svw{(1^{n-r})}}, \text{ and for } s>1, F^\perp_{\svw{(1^s)}}(\ncri_{\svw{(1^n)}})=0.\]

\item \cite[Proposition 3.16 and Corollary 3.18]{BBSSZ2014} 
\[\nch_\beta=\sum_{\alpha_{\ge_\ell} \beta} K_{\alpha, \beta}\, \nci_\alpha 
\iffpsi \nce_\beta=\sum_{\alpha_{\ge_\ell} \beta} K_{\alpha, \beta} \,\ncri_\alpha\]
and \svw{by Theorem~\ref{the:KasL}}
\[\nch_\beta=\sum_{\alpha } K^*_{\alpha,\beta}\,\ncri_\alpha \iffpsi \nce_\beta = \sum_{\alpha} K^*_{\alpha,\beta}\,\nci_\alpha.\]

\item \cite[Theorem 3.25]{BBSSZ2014} The ribbon function $\ncr_\beta$ expands positively in both immaculate bases:
\[\ncr_\beta=\sum_{\alpha\ge_\ell \beta} L_{\alpha, \beta} \nci_\alpha
\iffpsi
\ncr_{\beta^c}=\sum_{\alpha\ge_\ell \beta} L_{\alpha, \beta} \ncri_\alpha.\]
\sheilaMar{(Recall that $L_{\alpha,\beta}$ denotes the number of \emph{standard} immaculate tableaux of shape $\alpha$ and descent composition $\beta$.)}

\item \label{schurtodI} \cite[Theorem 3.38]{BBSSZ2014} The Schur function $s_\lambda$ with $\ell(\lambda)=k$ expands into the dual immaculate and row-strict dual immaculate bases as follows:
\[s_\lambda = \sum_{\sigma \in S_k} (-1)^{\svw{\sgn(\sigma)}}\dI_{\sigma(\lambda)}\iffpsi
s_{\lambda'} = \sum_{\sigma \in S_k} (-1)^{\svw{\sgn(\sigma)}}\rdI_{\sigma(\lambda)}\]
\sheilaMar{
where, 
for $\lambda$ a partition and $\sigma \in \svwMar{S_{k}}$, we define $\sigma(\lambda) =  (\lambda_{\sigma_1}+1-\sigma_1,\ldots, \lambda_{\sigma_k}+k-\sigma_k)$ provided $\lambda_{\sigma_i}+i-\sigma_i>0$ for each $i$.  If the latter condition is not satisfied, we define 
$\dI_{\sigma(\lambda)}=0=\rdI_{\sigma(\lambda)}.$}

\item \cite[Theorem 1.1]{AHM2018} For $\alpha$ a composition and $c_{\alpha\beta}\geq 0$, 
\[\dI_\alpha = \sum_{\beta} c_{\alpha\beta} \qsy_\beta\iffpsi \rdI_\alpha = \sum_{\beta} c_{\alpha\beta} \rqsy_\beta,\]
where $\qsy$ and $\rqsy$ are the Young quasisymmetric Schur and Young row-strict quasisymmetric Schur functions.

\end{enumerate}
\end{theorem}
\section{Skew row-strict dual immaculate functions}

Following the work of Berg et. al.~\cite{BBSSZ2014}, we define the poset $\mathfrak{P}$ of immaculate tableaux. The labelled poset $\mathfrak{P}$  is on the set of all compositions. Place an arrow from $\alpha$ to $\beta$ if $\beta\subset \alpha, $  and \svwMar{$|\alpha|-|\beta|=1$, denoted by} $\beta \subset_1 \alpha$. The label of $m$ on each cover $\alpha \overset{m}\longrightarrow \beta$ denotes the row containing the single additional box. Denote \svwMar{a} path from $\alpha$ to $\beta$ in $\mathfrak{P}$ by $P=[\alpha, \beta]$.
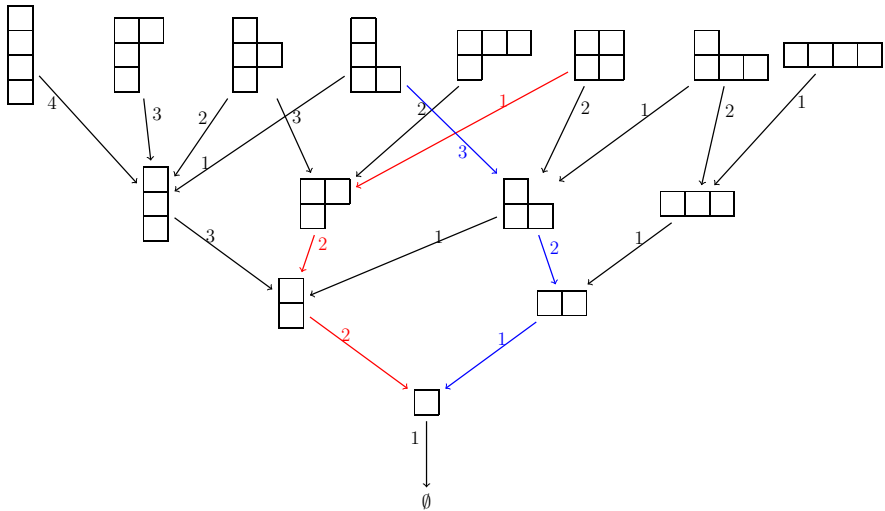
\begin{figure}[htb] \centering
 
		\scalebox{.6}{			
			\begin{tikzpicture}
			\newcommand*{\xdist}{*3}
			\newcommand*{\ydist}{*2.2}
\node (n0) at (0.00\xdist,0\ydist) 
			{
				$\emptyset$ 
			}; 
\node (n1) at (0,1\ydist) 
			{
				$\tableau{{}}$
			}; 
\node (n21) at (-1\xdist,2\ydist) 
			{
				$\tableau{{}\\{}}  $ 
			}; 

\node (n22) at (1\xdist,2\ydist) 
			{
				$\tableau{ {}&{}}$
			}; 
\node (n31) at (-2\xdist,3\ydist) 
			{
				$\tableau{ {}\\{}\\{}}$
			}; 
\node (n32) at (-.75\xdist,3\ydist) 
			{ 
				$\tableau{ {}&{}\\{}}$
			}; 
\node (n33) at (.75\xdist,3\ydist) 
			{
				$\tableau{ {}\\{}&{}}  $
			}; 
\node (n34) at (2\xdist,3\ydist) 
			{
				$\tableau{ {}&{}&{}}$
			}; 
\node (n41) at (-3\xdist,4.5\ydist) 
			{
				$\tableau{ {}\\{}\\{}\\{}}  $
			}; 
\node (n42) at (-2.125\xdist,4.5\ydist) 
			{
				$\tableau{ {}&{}\\{}\\{}}  $ 
			}; 

\node (n43) at (-1.25\xdist,4.5\ydist) 
			{
				$\tableau{ {}\\{}&{}\\{}}$
			}; 
\node (n44) at (.5\xdist,4.5\ydist) 
			{
				$\tableau{ {}&{}&{}\\{}}  $ 
			}; 
\node (n45) at (1.275\xdist,4.5\ydist) 
			{
				$\tableau{ {}&{}\\{}&{}}  $ 
			}; 
\node (n46) at (-.375\xdist,4.5\ydist) 
			{
				$\tableau{ {}\\{}\\{}&{}}$
			}; 
\node (n47) at (2.25\xdist,4.5\ydist) 
			{
				$\tableau{ {}\\{}&{}&{}}$
			}; 
\node (n48) at (3\xdist,4.5\ydist) 
			{
				$\tableau{ 
				{}&{}&{}&{}}$};
\draw [thick, ->] (n1) -- (n0) node [near start, left] {1};

\draw [thick, ->,red] (n21) -- (n1) node [near start, right] {2};
\draw [thick, ->,blue] (n22) -- (n1) node [near start, left] {1};
 
\draw [thick, ->] (n31) -- (n21) node [near start, right] {3};
\draw [thick, ->,red] (n32) -- (n21) node [near start, right] {2}; 
\draw [thick, ->] (n33) -- (n21) node [near start, left] {1}; 
\draw [thick, ->,blue] (n33) -- (n22) node [near start, right] {2}; 
\draw [thick, ->] (n34) -- (n22) node [near start, left] {1};

\draw [thick, ->] (n41) -- (n31) node [near start, left] {4}; 
\draw [thick, ->] (n42) -- (n31) node [near start, right] {3}; 
\draw [thick, ->] (n43) -- (n31) node [near start, left]{2}; 
\draw [thick, ->] (n46) -- (n31) node [near end, left] {1}; 
\draw [thick, ->] (n43) -- (n32) node [near start, right] {3}; 
\draw [thick, ->] (n44) -- (n32) node [near start, left] {2}; 
\draw [thick, ->,red] (n45) -- (n32) node [near start, left] {1}; 
\draw [thick, ->] (n45) -- (n33) node [near start, right] {2}; 
\draw [thick, ->,blue] (n46) -- (n33) node [near end, left] {3}; 
\draw [thick, ->] (n47) -- (n33) node [near start, left] {1}; 
\draw [thick, ->] (n47) -- (n34) node [near start, right] {2}; 
\draw [thick, ->] (n48) -- (n34) node [right, near start] {1};
\end{tikzpicture}	
	}
 
\caption{\small{The start of the poset $\mathfrak{P}$ with edge labels.  A horizontal 3-strip is shown in red and a vertical 3-strip is shown in blue.}}\label{fig:skewposet}	
\end{figure}

To obtain a standard skew immaculate tableau from a path $P=[\alpha,\beta]$, for each $m_i$, $1\leq i \leq k$, label the rightmost unlabeled cell in row $m_i$ of $\alpha$ with $k-i+1$, see Example~\ref{ex:StandardSkew}. In order to understand the combinatorial models for skew dual immaculate and skew row-strict dual immaculate functions we define two special types of paths. 

\begin{definition}\label{def:skew-paths}A path $P=\{\alpha=\beta^{(0)}\overset{m_1}\longrightarrow\beta^{(1)}\overset{m_2}\longrightarrow \svw{\cdots}\overset{m_k}\longrightarrow\beta^{(k)}=\beta\}$ in the poset $\mathfrak{P}$ is a 
\begin{itemize}
\item {\em horizontal $k$-strip} if $m_1\leq m_2\leq \cdots \leq m_k$, and a
\item {\em vertical $k$-strip} if $m_1>m_2>\cdots >m_k$.
\end{itemize}
\end{definition}

The horizontal $3$-strip (red path) and vertical $3$-strip (blue path) in Figure~\ref{fig:skewposet} give rise to the \svw{following tableaux.}

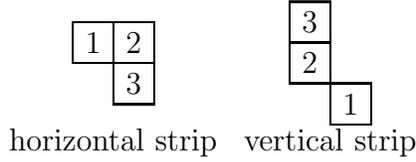
\begin{center}
\begin{tabular}{cc}
$\tableau{1&2\\&3} $& $\tableau{3\\2\\&1}$\\
horizontal strip & vertical strip
\end{tabular}
\end{center}
\sheila{
We can now make the following definition.
\begin{definition}\label{def:standard-skew-tableau}
A \emph{standard skew immaculate tableau} of shape $\alpha/\beta$ is a filling of the shape $\alpha/\beta$ with the distinct positive integers $\{1,2,\ldots,|\alpha/\beta|\}$, such that rows strictly increase from left to right and the labels in $\alpha/\beta$ in cells that are in the first column of $\alpha$ must increase from bottom to top.  
\end{definition}
}
For a path $P=[\alpha,\beta]$ of length $k$, define the {\em descent set of $P$} to be 
$D(P)=\{k-i:m_i>m_{i+1}\}$ and the {\em weak ascent set of $P$} to 
$A(P)=\{k-i:m_i\leq m_{i+1}\}$. 
\sheilaFeb{
Each such path $P=[\alpha,\beta]$  corresponds to a unique standard skew immaculate 
tableau $T$ of shape $\alpha/\beta$, and conversely.  Furthermore, 
the descent set $D(P)$ coincides with the descent set 
$\Des_{\dI}(T)=\{i: i+1 \text{ appears strictly above $i$ in $T$}\}$, 
and similarly the ascent set $A(P)$ coincides with 
the descent set $\Des_{\rdI}(T)=\{i: i+1 \text{ appears weakly below $i$ in $T$}\}$.}

\begin{example}\label{ex:StandardSkew} For $\alpha/\beta = (3,2,3)/(1,1,2)$, 
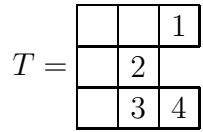
\[T=\tableau{{}&{}&1\\{}&2\\{}&3&4}\] is a valid standard skew immaculate tableau.  It corresponds to the path $P = (3,2,3) \overset{1}\rightarrow (2,2,3) \overset{1}\rightarrow (1,2,3) \overset{2}\rightarrow (1,1,3)\overset{3}\rightarrow (1,1,2)$.  Further, $\Des_{\rdI}(T) = \{1,2,3\}$, $D(P)$ is empty, and $A(P) = \{1,2,3\}$.  
\end{example}

Given a path $P=[\alpha, \emptyset]$ corresponding to a standard immaculate tableau $T$, we have that $\Des_{\dI}(T)=D(P)$ and $\Des_{\rdI}(T)=A(P)$, \svw{ by comparing the definitions, and is illustrated} in Figure~\ref{fig:descents}.

\begin{figure}[ht]
\begin{center}
\[T = \tableau{4&7\\2&3&5\\1&6}\]

\[P = (2,3,2)\overset{3}\rightarrow (2,3,1) \overset{1}\rightarrow (1,3,1)\overset{2}\rightarrow (1,2,1) \overset{3}\rightarrow (2,1)\overset{2}\rightarrow (1,1) \overset{2}\rightarrow (1)\overset{1}\rightarrow  \emptyset\]
\end{center}
\caption{The path $P$ has $D(P) = \{1,3,6\}$ and $A(P)=\{2,4,5\}$, while $\Des_{\dI}(T) = \{1,3,6\}$ and $\Des_{\rdI}(T) = \{2,4,5\}$. }\label{fig:descents}
\end{figure}

Note that given a skew immaculate tableau, it can be decomposed into horizontal or vertical strips in several ways.  An example of decomposing a tableau into either horizontal or vertical strips is given in Figure~\ref{fig:horizvertdecomp}.

\begin{figure}[ht]
\[T=\tableau{2&4&5\\&3\\&&1} \]\[ P=(3,2,3)\overset{3}\rightarrow (3,2,2)\overset{3}\rightarrow (3,2,1)\overset{2}\rightarrow (3,1,1) \overset{3}\rightarrow (3,1)\overset{1}\rightarrow (2,1)\]

\caption{The standard skew immaculate tableau $T$ and its corresponding path can be decomposed into maximal horizontal strips $(3,2,3)\overset{3}\rightarrow(3,2,2)\overset{3}\rightarrow (3,2,1)$, $(3,2,1)\overset{2}\rightarrow (3,1,1)\overset{3}\rightarrow (3,1)$, and $(3,1)\overset{1}\rightarrow (2,1)$.  Alternatively, decompose $P$ into maximal vertical strips $(3,2,3)\overset{3}\rightarrow (3,2,2)$, $(3,2,2)\overset{3}\rightarrow (3,2,1)\overset{2}\rightarrow (3,1,1)$, and $(3,1,1)\overset{3}\rightarrow (3,1)\overset{1}\rightarrow (2,1)$.}\label{fig:horizvertdecomp}
\end{figure}

\svw{ 
In~\cite{BBSSZ2014} the poset $\mathfrak{P}$ and horizontal strips are used to define the skew dual immaculate functions as follows.}

\begin{definition}\label{def:BBSSZ2014-skewdI-tab} For $\{\gamma: \beta\subseteq \gamma \subseteq \alpha\}$ an interval in~$\mathfrak{P}$, define the \emph{skew dual immaculate function} to be 
\[\dI_{\alpha/\beta}=\sum_\gamma \langle \mathfrak{S}_\beta \mathbf{h}_\gamma, \dI_\alpha\rangle M_\gamma.\]
\end{definition}

This can be rewritten in terms of both the fundamental basis and the dual immaculate basis. 

\begin{prop}{\cite[Propositions 3.47 and 3.48]{BBSSZ2014}}\label{prop:skewdI}
For $\{\gamma: \beta\subseteq \gamma \subseteq \alpha\}$ an interval in~$\mathfrak{P}$, 
\begin{align}\dI_{\alpha/\beta} & = \sum_\gamma \langle \Imm_\beta \ncr_\gamma, \dI_\alpha\rangle F_\gamma \\
&= \sum_{\gamma} \langle \Imm_\beta \Imm_\gamma,\dI_\alpha\rangle \dI_\gamma\\
&=\sheilaFeb{\sum_{P=[\beta,\alpha]\in \mathfrak{P}} F_{\comp(D(P))}
=\sum_{\substack {T \text{ a standard skew immaculate} 
                              \\ \text{ tableau of shape  }  \alpha/\beta}}
 F_{\comp(\Des_{\dI}(T))};}
\end{align}
in the last line, each path $P$ from $\beta$ to $\alpha$ corresponds to
a unique standard skew immaculate tableau $T$ of shape $\alpha/\beta$.
\end{prop}

Note that the number of standard skew immaculate tableaux  $T$ of shape $\alpha/\beta$ with $\comp(\Des_{\dI}(T)) =\gamma$ is $\langle \Imm_\beta \ncr_\gamma,\dI_\alpha\rangle$.  

\begin{definition}\label{def:skew-row-strict-fn}
For $\{\gamma: \beta\subseteq \gamma \subseteq \alpha\}$ an interval in~$\mathfrak{P}$, define the \emph{skew row-strict dual immaculate function} to be \[\RI_{\alpha/\beta}=\sum_{\gamma}\langle\rI_{\beta}\nch_{\gamma},\RI_{\alpha}\rangle M_{\gamma}.\]
\end{definition}

We now quickly obtain the following.
\begin{theorem}\label{the:skewRI}
For $\{\gamma: \beta\subseteq \gamma \subseteq \alpha\}$ an interval in~$\mathfrak{P}$,  
\begin{align}
\RI_{\alpha/\beta}
& =\sum_{\gamma}\langle\rI_{\beta}\ncr_{\gamma},\RI_{\alpha}\rangle F_{\gamma} \\
&=\psi(\dI_{\alpha/\beta}) \\
& =\sum_{\gamma}\langle\rI_{\beta}\rI_{\gamma},\RI_{\alpha}\rangle \RI_{\gamma}\\
&= \sheilaFeb{\sum_{P=[\beta,\alpha]\in \mathfrak{P}}F_{\comp(A(P))}
=\sum_{\substack {T \text{ a standard skew immaculate} 
                              \\ \text{ tableau of shape  }  \alpha/\beta}} 
F_{\comp(\Des_{\rdI}(T))}.}
\end{align}
\end{theorem}
\begin{proof} The first equality is immediate from Definition~\ref{def:skew-row-strict-fn} by using~\eqref{eq:ribbon-to-homogeneous-nsym} to expand  $\nch_\gamma$ in terms of the ribbon basis, interchanging the order of summation, and finally using~\eqref{eq:FtoM}:
\[\RI_{\alpha/\beta}= \sum_\gamma \langle\rI_{\beta} \sum_{\tau \succcurlyeq \gamma} \ncr_{\tau},\RI_{\alpha}\rangle M_{\gamma}
=\sum_\tau \langle\rI_{\beta}  \ncr_{\tau},\RI_{\alpha}\rangle \left(\sum_{ \gamma  \preccurlyeq \tau }M_{\gamma}\right)
=\sum_\tau \langle\rI_{\beta}  \ncr_{\tau},\RI_{\alpha}\rangle F_{\tau}.\]
The second line then follows by applying $\psi$ to the first equality in Proposition~\ref{prop:skewdI}, and using the invariance of the pairing under $\psi$, which gives 
\[\psi(\dI_{\alpha/\beta})=\sum_\gamma \langle \psi(\Imm_\beta)\ \psi(\ncr_\gamma),  \psi(\dI_\alpha)\rangle\ \psi(F_\gamma)
=\sum_\gamma \langle \RI_\beta\ \ncr_{\gamma^c}, \RI_\alpha\rangle\ F_{\gamma^c},
\]
where we have used~\eqref{eqn:psidualimm}, \eqref{eqn:psinsym} and~\eqref{eqn:psiF}.
The last two lines are now immediate by applying $\psi$ to the last two equations in Proposition~\ref{prop:skewdI}, since $A(P)$ and $D(P)$ are complementary by definition, \sheilaFeb{and each path $P$ from $\beta$ to $\alpha$ corresponds to
a unique standard skew immaculate tableau $T$ of shape $\alpha/\beta$.}
\end{proof}

 \begin{definition}\label{def:skewimmtab}
Let $\alpha$ and $\beta$ be compositions with $\beta \subseteq \alpha$.  Then a filling $T$ of the diagram of $\alpha/\beta$ with positive integers is a {\em skew immaculate tableau} provided
\begin{enumerate}
\item the entries in the first column of $\alpha$ (if any remain in $\alpha/\beta$) are strictly increasing from bottom to top, and 
\item rows weakly increase from left to right.  
\end{enumerate}
Similarly, $T$ is a {\em skew row-strict immaculate tableau} if 
\begin{enumerate}
\item the entries in the first column of $\alpha$ (if any remain in $\alpha/\beta$) are weakly increasing from bottom to top, and 
\item rows strictly increase from left to right.  
\end{enumerate}
\end{definition}
We now have the needed interpretation of the coefficients in Definitions~\ref{def:BBSSZ2014-skewdI-tab} and~\ref{def:skew-row-strict-fn} to rewrite $\dI_{\alpha/\beta}$ and $\rdI_{\alpha/\beta}$ as generating functions of skew immaculate tableaux.  
\begin{theorem}\label{thm:skew-semistandard-dualImm} Let $\alpha$ and $\beta$ be compositions with $\beta \subseteq \alpha$.  Then 
\[\dI_{\alpha/\beta} = \sum_{T}x^T\]
where the sum is over all skew immaculate tableaux of shape $\alpha/\beta$, and 
\[\rdI_{\alpha/\beta} = \sum_{T} x^T\]
where the sum is over all skew row-strict immaculate tableaux of shape $\alpha/\beta$.  
\end{theorem}
\begin{proof} \svw{By Point \eqref{thm:pieri2}} in Theorem~\ref{thm:psiresults}, we know that for $\gamma=\gamma_1\gamma_2\svw{\cdots}\gamma_k$, $\alpha$ can be obtained from $\beta$ by a series of vertical strips of lengths $\gamma_1,\gamma_2,\dots,\gamma_k$. Thus the coefficient $\langle\rI_{\beta}\nch_{\gamma},\RI_{\alpha}\rangle$ represents the number of ways to add a sequence of vertical strips of lengths $\gamma_1,\gamma_2,\dots,\gamma_k$ from $\beta$ to $\alpha$, which counts the number of skew immaculate tableaux $T$ of shape $\alpha/\beta$ such that the descent composition of $T$ is coarser than $\gamma$, since adding a vertical strip after another one may or may not create a descent. \sheila{See Example~\ref{ex:for-thm-skew-semistandard} below.} Thus
\[ \langle \mathfrak{S}_\beta \mathbf{h}_\gamma, \dI_\alpha\rangle \]
is the number of skew immaculate tableaux of shape $\alpha/\beta$ of content $\gamma$ and 
\[\langle\rI_{\beta}\nch_{\gamma},\RI_{\alpha}\rangle \]
is the number of skew row-strict immaculate tableaux of shape $\alpha/\beta$ 
of content $\gamma$. 
The result now follows immediately from the definitions.
\end{proof}

\begin{example}\label{ex:for-thm-skew-semistandard}
Consider 
\[T=\tableau{1&4\\&3\\&2} \]
and corresponding path \[P= \svw{(2,2,2)\overset{3}\rightarrow(2,2,1)\overset{2}\rightarrow(2,1,1)\overset{1}\rightarrow (1,1,1) \overset{3}\rightarrow(1,1).}\]

 Note that $T$ can be considered to be formed from vertical strips corresponding to $\gamma = (1,3)$ or $(1,1,2)$, or $(1,2,1)$ or $(1,1,1,1)$ since $\comp(\Des_{\rdI}(T))=(1,3)$ and is coarser than the listed options for $\gamma.$ 
\end{example}

\sheila{In analogy with the well-known Schur function identity \cite[\svwMar{Chapter 1} Eqn. (5.9)]{MACD1995}, 
\begin{equation}\label{eqn:sheila-Schur-two-vars-1}s_\lambda(X,Y)= \sum_{\mu\subset \lambda}s_\mu (X)  s_{\lambda/\mu} (Y),\end{equation}
Theorem~\ref{thm:skew-semistandard-dualImm} immediately gives us the following:
\begin{theorem}
Suppose we have two sets of variables, $X$ and $Y,$ ordered so that the alphabet $X$  precedes the alphabet $Y$.  Then 
\begin{equation}\label{eqn:sheila-dI-two-vars}\dI_\alpha(X,Y)= \sum_{\beta\subset \alpha}\dI_\beta (X)  \dI_{\alpha/\beta} (Y),\end{equation}
and 
\begin{equation}\label{eqn:sheila-rdI-two-vars}\rdI_\alpha(X,Y)= \sum_{\beta\subset \alpha}\rdI_\beta (X)  \rdI_{\alpha/\beta} (Y).\end{equation}
\end{theorem}
}
\subsection{Hopf algebra approach}
We consider the Hopf algebra approach to defining skew dual immaculate functions and establish that it is equivalent to the previous definition.  To start, we provide a brief introduction to the necessary Hopf algebra background.  

\svw{We have that} $\NSym$ and $\QSym$ form dual Hopf algebras using the pairing $\langle \cdot,\cdot \rangle: \NSym \otimes \QSym \rightarrow \mathbb{Q}$ defined by $\langle \nch_\alpha, M_\beta\rangle = \delta_{\alpha\beta}$ where $\delta_{\alpha\beta} = 1 $ if $\alpha = \beta$ and 0 otherwise.  

Given dual bases $\{B_i\}_{i \in I}$ and $\{D_i\}_{i\in I}$, 
\begin{align*}
B_i\cdot B_j = \sum_{k} b_{i,j}^k B_k &\quad \Leftrightarrow\quad  \Delta D_k = \sum_{i,j} b_{i,j}^k D_i \otimes D_j\\
D_i \cdot D_j =\sum_{k} d_{i,j}^k D_k &\quad \Leftrightarrow\quad  \Delta B_k = \sum_{i,j} d_{i,j}^k B_i \otimes B_j
\end{align*}
where $\cdot$ is the product and $\Delta$ is the coproduct. 

For the fundamental quasisymmetric functions, we have that 
\begin{equation}\label{eq:fundcoprod}
\Delta F_\alpha = \sum_{\substack{(\beta,\gamma) \text{ with }\\\beta\cdot \gamma = \alpha \text{ or}\\\beta\odot \gamma=\alpha}} F_\beta\otimes F_\gamma
\end{equation}
where for $\beta = (\beta_1,\ldots, \beta_k)$ and $\gamma = (\gamma_1,\ldots, \gamma_n)$, $\beta\cdot\gamma = (\beta_1,\ldots, \beta_k,\gamma_1,\ldots, \gamma_n)$ is the \svw{\emph{concatenation}} of $\beta$ and $\gamma$,  and $\beta\odot\gamma = (\beta_1, \ldots, \beta_{k-1},\beta_k+\gamma_1,\gamma_2,\ldots, \gamma_n)$ is the \svw{\emph{near-concatenation}} of $\beta$ and $\gamma$. 

Following \svw{\cite{BLvW2011},} we can define the coproduct \svw{$\Delta\I^*_\alpha$} in terms of skew elements $\widetilde{\dI_{\alpha/\gamma}}$.
\begin{definition}\label{def:skewdIhopf} Let $\alpha \vDash n$. \svwMar{Then we define}
\[\svw{\Delta\dI_\alpha}=\emn{\sum_{\gamma} \dI_\gamma\otimes\widetilde{\dI_{\alpha/\gamma}}}.\]
\end{definition}

We show that $\widetilde{\dI_{\alpha/\gamma}}=\dI_{\alpha/\gamma}$ as described in Proposition~\ref{prop:skewdI}.

\begin{lemma}\label{lem:skewdi}
\[\widetilde{\dI_{\alpha/\gamma}}=\dI_{\alpha/\gamma}=\sum_{T}F_{\comp(\Des_{\dI}(T))}\]
where the sum is over all \svw{standard skew} immaculate tableaux $T$ of shape $\alpha/\gamma$.
\end{lemma}

\begin{proof}
We use the technique of~\cite[Proposition 3.1]{BLvW2011}. Let $T$ be a \svw{standard skew} immaculate tableaux such that $|T|=n$. For any $k$ with $0\leq k\leq n$, let $\mho_k(T)$ be the standardization of the skew tableaux consisting of cells of $T$ with entries $\{n-k+1,\svw{\ldots,} n\}$. Also let $\Omega_k(T)$ be the skew tableaux consisting of the cells of $T$ after removing the entries $\{k+1,\svw{\ldots,} n\}$ as in Figure~\ref{fig:omskew}.

\begin{figure}
\[ T=\tableau{4&5&8\\*&*&6&7\\*&*&2&3\\*&1&9}\emn{\quad \Omega_4(T) = \tableau{4\\*&*\\*&*&2&3\\*&1}\quad \mho_5(T)=\tableau{*&1&4\\*&*&2&3\\*&*&*&*\\*&*&5} } \]
\caption{An example of \emn{$\Omega_{n-k}(T)$ and $\mho_k(T)$}.}\label{fig:omskew}
\end{figure}

Note that if $T$ is a standard immaculate tableau of shape $\alpha$, then $T=\Omega_{n-k}(T)\cup(\mho_k(T)+(n-k))$ where $\mho_k(T)+(n-k)$ is $\mho_k(T)$ with $n-k$ added to each entry.  Suppose $\Des_{\dI}(T) = \alpha$ with $|\alpha|=n$.  Then we can rewrite~\eqref{eq:fundcoprod} as 
\[\Delta F_\alpha = \sum_{i=0}^n F_{\beta_i}\otimes F_{\gamma_i}\] 
where $|\beta_i|=\svw{n-i}$, $|\gamma_i|=\svw{i}$, and either $\beta_i\cdot\gamma_i=\alpha$ or $\beta_i\odot\gamma_i=\alpha$. \emn{Observe} that $\beta_i =\comp( \Des_{\dI}(\emn{\Omega_{n-i}}(T)))$ and $\gamma_i = \comp(\Des_{\dI}(\emn{\mho_{i}}(T)))$.

Then
\begin{align*}
\Delta\dI_\alpha
&=\Delta\left(\sum_T F_{\comp(\Des_{\dI}(T))}\right)\\
&=\sum_{T }\Delta F_{\comp(\Des_{\dI}(T))} \\
&=\sum_{T}\sum^n_{ i=0}F_{\beta_i}\otimes F_{\gamma_i}
\end{align*}
where $T$ is a standard \svw{immaculate} tableau of shape $\alpha$.  

Further, by Definition~\ref{def:skewdIhopf} we have
\begin{align*}
\Delta \dI_\alpha &=
\sum_\delta \emn{\dI_\delta\otimes \widetilde{\dI_{\alpha/\delta}}}\\ & = \sum_{\delta}\emn{ \sum_S F_{\comp(\Des_{\dI}(S))}\otimes \widetilde{\dI_{\alpha/\delta}}  }
\end{align*} where $S$ is a standard immaculate tableau of shape $\delta$. 

For a fixed $S$ of shape $\delta$ with $|\delta|=n-k$ for some $k$, there exists a standard immaculate tableau $T$ of shape $\alpha$ such that $S=\Omega_{n-k}(T)$. Then $\mho_k(T)$ has shape $\alpha/\delta$.  Similarly, given a standard immaculate tableau $T$ of shape $\alpha$, $T=\emn{\Omega_{n-k}(T)\cup(\mho_{k}(T)+(n-k))}$ where $\Omega_{n-k}(T)$ has shape $\delta$ with $|\delta|=n-k$ and $\mho_{k}(T)$ has shape $\alpha/\delta$. Thus 
\[\widetilde{\dI_{\alpha/\delta}} = \sum_{T} F_{\comp(\Des_{\dI}(T))}=\dI_{\alpha/\delta}\]
where $T$ is a \svw{standard skew} immaculate tableau of shape $\alpha/\delta$.
\end{proof}
\svwMar{By the above lemma and Theorem~\ref{the:skewRI}, we have the following.}
\svwMar{\begin{theorem}\label{the:coproduct} Let $\alpha$ be a composition. Then 
\[\Delta \dI_\alpha = \sum_{\beta}  {\dI_\beta\otimes\dI_{\alpha/\beta}}
\quad \text{ and }\quad 
  \Delta \rdI_\alpha = \sum_{\beta}  {\rdI_\beta\otimes\rdI_{\alpha/\beta}}.\]  
\end{theorem}}
\subsection{Expansions of skew Schur functions}
We can also use a Hopf algebra approach to establish skew versions of \svw{Point} \eqref{schurtodI} in Theorem~\ref{thm:psiresults}, \svw{from where we recall that for $\lambda$ a partition and $\sigma \in S_{\ell(\lambda)}$, define $\sigma(\lambda) =  (\lambda_{\sigma_1}+1-\sigma_1,\ldots, \lambda_{\sigma_k}+k-\sigma_k)$ provided $\lambda_{\sigma_i}+i-\sigma_i>0$ for each $i$.} 

\svw{Also} recall that $s_{\lambda/\mu}=\det(h_{\lambda_i-\mu_j-i+j})$. If we consider compositions $\alpha \subseteq \lambda$, we can define $s_{\lambda/\alpha} = \det(h_{\lambda_i - \alpha_j-i+j})$.  Note that if there exists some $\alpha_j - j=\alpha_k-k$ for some $j\neq k$, $s_{\lambda/\alpha}=0$ since two columns of the matrix will be equal.  If no such pair $j,k$ exists, then there exists a unique permutation $\tau$ such that 
\svw{$\tau(\alpha) = (\alpha_{\tau_1}+1-\tau_1,\ldots, \alpha_{\tau_k}+k-\tau_k)=\mu$} where $\mu$ is a partition.  In this case, 
\begin{equation}\label{eqn:skewschur}
s_{\lambda/\mu}=(-1)^{\svw{\sgn(\tau)}} s_{\lambda/\alpha}.
\end{equation}
\begin{theorem}\label{thm:skewschurexpansion}
Let $\lambda$ and $\mu$ be partitions with $\mu \subseteq \lambda$.  Then 
\[s_{\lambda/\mu} = \sum_{\sigma \in S_{\ell(\lambda)}} (-1)^{\sgn(\sigma)+\sgn(\tau)} \dI_{\sigma(\lambda)/\tau(\mu)}\] for any choice of $\tau$ such that $\tau(\mu)$ is a composition.
\end{theorem}
 
\begin{proof}
Recall that \svw{$\Delta(s_\lambda) = \sum_{\mu} s_{\lambda/\mu} \otimes s_{\mu} = \sum_{\mu} s_{\mu} \otimes s_{\lambda/\mu}$ because the Hopf algebra of symmetric functions is cocommutative.}  We can rewrite $\Delta(s_\lambda)$ using Theorem~\ref{thm:psiresults}, \svw{Point} \eqref{schurtodI}.  Then 
\begin{align*}
\Delta(s_\lambda) & = \Delta \left(\sum_{\sigma \in S_{\ell(\lambda)}} (-1)^{\sgn(\sigma)}\dI_{\sigma(\lambda)}\right)\\
&= \sum_{\sigma \in S_{\ell(\lambda)}}(-1)^{\sgn(\sigma)} \Delta \dI_{\sigma(\lambda)}\\
&= \sum_{\sigma \in S_{\ell(\lambda)}}(-1)^{\sgn(\sigma)} \left(\sum_{\beta} \emn{\dI_{\beta}\otimes\dI_{\sigma(\lambda)/\beta}}\right)\\
&= \sum_{\beta}\emn{\dI_{\beta} \otimes \left(\sum_{\sigma \in S_{\ell(\lambda)}} (-1)^{\sgn(\sigma)}\dI_{\sigma(\lambda)/\beta}\right)}.
\end{align*}

On the other hand, 
\begin{align*}
\sum_{\mu}\emn{s_{\mu}\otimes s_{\lambda/\mu}} &= \sum_\mu \emn{ \left(\sum_{\tau \in S_{\ell(\mu)}} (-1)^{\sgn(\tau)} \dI_{\tau(\mu)}\right)\otimes s_{\lambda/\mu}}\\
&=\sum_{\mu} \sum_{\tau \in S_{\ell(\mu)}} (-1)^{\sgn(\tau)} \left(\emn{\dI_{\tau(\mu)}\otimes s_{\lambda/\mu}}\right)\\
&=\svw{\sum_{\beta}}\emn{\dI_{\beta}\otimes \left(\sum_{\tau \in S_{\ell(\beta)}} (-1)^{\sgn(\tau)}s_{\lambda/\tau^{-1}(\beta)}\right)}\\
\end{align*}
where $\beta$ is a composition and $\tau^{-1}(\beta)$ is a partition.
Thus for a fixed choice of $\beta$,
\[\sum_{\sigma\in S_{\ell(\lambda)}}(-1)^{\sgn(\sigma)}\dI_{\sigma(\lambda)/\beta} = \sum_{\tau \in S_{\ell(\beta)}}(-1)^{\sgn(\tau)} s_{\lambda/\tau^{-1}(\beta)}.\]
Note that for each $\beta$, there is at most one $\tau\in S_{\ell(\beta)}$ such that $s_{\lambda/\tau^{-1}(\beta)}=s_{\lambda/\mu}\neq 0$ for a partition $\mu$.  Thus
\[s_{\lambda/\mu} = \sum_{\sigma \in S_{\ell(\lambda)}} (-1)^{\sgn(\sigma)+\sgn(\tau)} \dI_{\sigma(\lambda)/\tau(\mu)}\] for any valid choice of $\tau$.
\end{proof}

Choosing $\tau$ as the identity gives the following corollary.
\begin{corollary}\label{cor:schurexpdI}
For partitions $\lambda$ and $\mu$ with $\mu \subseteq \lambda$, 
\[s_{\lambda/\mu} = \sum_{\sigma\in S_{\ell(\lambda)}} (-1)^{\sgn(\sigma)}\dI_{\sigma(\lambda)/\mu}.\]
\end{corollary}

Applying $\psi$ to both sides of Theorem~\ref{thm:skewschurexpansion} gives us an expansion in terms of the row-strict dual immaculate functions. 

\begin{corollary}\label{cor:schurexprsdI}
For partitions $\lambda$ and $\mu$ with $\mu \subseteq\lambda$ and $\tau \in S_{\ell(\mu)}$ such that $\tau(\mu)$ is a composition, 
\[s_{\lambda'/\mu'} = \sum_{\sigma \in S_{\ell(\lambda)}} (-1)^{\sgn(\sigma)+\sgn(\tau)} \rdI_{\sigma(\lambda)/\tau(\mu)}.\]
\end{corollary}

\section{Hook dual immaculate functions}\label{sec:hookdI}
Now that we have skew row-strict dual immaculate functions, we can define {\em hook dual immaculate functions} in a combinatorial manner analogous to the hook Schur functions~\cite{rem1984} and hook quasisymmetric Schur functions~\cite{MN2018}.  
\begin{definition}  Let $\mathcal{A} = \{1,2,\ldots, \ell\}$ and $\mathcal{A}'=\{1',2',\ldots,k'\}$ be two alphabets with $1<2<\cdots<\ell<1'<2'<\cdots<k'$.  Then a {\em semistandard hook immaculate tableau of shape $\alpha$} is a filling of the diagram of $\alpha$ with entries from $\mathcal{A}\cup \mathcal{A}'$ such that 
\begin{enumerate}
    \item the first column increases from bottom to top with the increase strict in $\mathcal{A}$ and weak in $\mathcal{A}',$ and
    \item each row increases from left to right, weakly in $\mathcal{A}$ and strictly in $\mathcal{A}'.$
\end{enumerate}
Denote the set of all semistandard hook immaculate tableaux of shape $\alpha$ by $HI_\alpha$. 

The content monomial of a \svwMar{semistandard hook immaculate tableau} $T$ is a monomial in two alphabets, $x_1,\ldots,x_\ell$ and $y_1,\ldots,y_k$, where 
\[z^T= \prod_{i\in \mathcal{A}\cup\mathcal{A'}} z_i^{\text{\# of $i$'s in $T$}}\] where $z_i = x_i$ if $i \in \mathcal{A}$ and $z_i = y_i$ if $i \in \mathcal{A}'$.
\end{definition}

\begin{example}
Let $\alpha = (3,1,2,4,3)$. Then $T$, as shown below, is a \svwMar{semistandard hook immaculate tableau} with content monomial \svw{$z^T=x_1^2x_2x_3^2y_1^3y_2y_3y_4^2y_5$.}

    \[T=\tableau{1'&2'&4'\\1'&3'&4'&5'\\3&1'\\2\\1&1&3}\]
    
\end{example}

\begin{definition}\label{def:hookdI}
The \svwMar{\emph{hook dual immaculate function}} indexed by $\alpha$ is 
\[\hdI_\alpha(X,Y)=\hdI_\alpha(x_1,\ldots,x_l,y_1,\ldots,y_k) = \sum_{T\in HI_\alpha}z^T.\]
\end{definition}

It follows immediately from the definition that 
\begin{equation} \label{hookdualimm}
\hdI_\alpha(X,Y) = \sum_{\gamma \subseteq \alpha} \dI_\gamma(X)\rdI_{\alpha/\gamma}\svw{(Y).}
\end{equation}

We can also expand $\hdI_\alpha(X,Y)$ in terms of the \svwMar{{super fundamental quasisymmetric functions}.}  We use the definition in~\cite{MN2018}. 
\begin{definition}\label{def:superfund} For $\alpha \vDash n$, the \svwMar{\emph{super fundamental quasisymmetric function} indexed by $\alpha$ is}
\[\tilde{Q}_\alpha(X,Y) = \sum_{\substack{a_1\leq a_2\leq \cdots \leq a_n\\ a_i = a_{i+1} \in \mathcal{A} \Rightarrow i \notin \set(\alpha)\\
a_i=a_{i+1} \in \mathcal{A}' \Rightarrow i \in \set(\alpha)}} z_{a_1}z_{a_2}\cdots z_{a_n},\]
where $z_a = x_a$ if $a \in \mathcal{A}$ and $z_{a'} = y_a$ for $a' \in \mathcal{A}'$.  
\end{definition}

\begin{theorem}{\cite[Theorem 4.1]{MN2018}}\label{thm:superfund} For $\alpha \vDash n$, 
\[\tilde{Q}_\alpha(X,Y) = \sum_{i=0}^n F_\beta(X) F_\gamma(Y)\]
where $\beta\cdot \gamma = \alpha$ if $i \in \set(\alpha)$ and $\beta\odot\gamma=\alpha$ if $i \notin \set(\alpha)$. 
\end{theorem}

\svwMar{As usual, we must have a standardization procedure for semistandard hook immaculate tableaux and an appropriate descent set to index the super fundamental quasisymmetric functions.  To {\em standardize} a semistandard hook immaculate tableau $T$, first replace the entries of $T$ from $\mathcal{A}$ by scanning unprimed entries from left to right, starting with the top row, replacing $1$'s as they are encountered by $1,2, \ldots$ in this reading order, followed by $2$'s, etc. Next continue with the entries of $\mathcal{A}'$ by scanning from right to left starting with the bottom row. The result is denoted by $\stdz(T)$.}

\begin{example}\label{ex:hookstandardize} The reading \svwMar{order} of $T$, as shown below, is $3,2,1,1,3,1',5',4',3',1',4',2', \svwMar{1'}$, giving rise to $\stdz(T)$ below.
\[T=\tableau{1'&2'&4'\\1'&3'&4'&5'\\3&1'\\2\\1&1&3} \qquad \stdz(T) = \tableau{8&9&12\\7&10&11&13\\4&6\\3\\1&2&5}\]
 \end{example}

Note that the standardization of a \svwMar{hook immaculate} tableau is a standard \svwMar{immaculate} tableau.  Recall that the descent set of a standard  \svwMar{immaculate} tableau $S$ is $\Des_{\dI}(S) = \{i:i+1$ is strictly above $i$ in $S\}$. The descent set for $\stdz(T)$ in Example~\ref{ex:hookstandardize} is $\Des_{\dI}(\stdz(T))=\svw{\{2,3,5,6, 7, 11\}}$.  From the definition of standardization, we note that if $T$ is a hook immaculate tableau of shape $\alpha$ with $T=S\cup U$ where $S$ is an immaculate tableau of shape $\beta$ and $U$ is a skew row-strict immaculate tableau of shape $\alpha/\beta$, then \[\Des_{\dI}(\stdz(T)) = \Des_{\dI}(\stdz(S)) \cup (\Des_{\rdI}(\stdz(U))^c+|\beta|)\]
if $|\beta|+1$ is weakly lower than $|\beta|$ in $\svw{\stdz(T)}$ and 
\[\Des_{\dI}(\stdz(T)) = \Des_{\dI}(\stdz(S)) \cup (\Des_{\rdI}(\stdz(U))^c+|\beta|)\cup \{|\beta|\}\]
if $|\beta|+1$ appears strictly above $|\beta|$ in $\stdz(T)$.

\begin{theorem}\label{thm:hookfundexp}
Let $\alpha \vDash n$.  Then 
\[\hdI_\alpha(X,Y) = \sum_{S} \tilde{Q}_{\comp(\Des_{\dI}(S))}(X,Y)\]
where the sum is over all standard \svwMar{immaculate} tableaux of shape $\alpha$. 
\end{theorem}
\begin{proof}  We show that each polynomial consists of the same monomials.  Suppose\\ $x_{a_1}\cdots x_{a_k}y_{b_1}\cdots y_{b_m}$ is the content monomial associated with a hook immaculate tableau $T$ of shape $\alpha$ with $a_1\leq a_2\leq \cdots \leq a_k$ and $b_1\leq b_2\leq \cdots \leq b_m$.  Note that if $a_i = a_{i+1}$, then $i \notin \Des_{\dI}(\stdz(T))$ by the standardization procedure.  Similarly, if $b_i = b_{i+1}$, $i+k \in \Des_{\dI}(\stdz(T))$, since $b_i'$ must occur in a lower \svw{row} of $T$ than $b_{i+1}'$.  Thus $x_{a_1}\cdots x_{a_k}y_{b_1}\cdots y_{b_m}$ is a monomial in $\tilde{Q}_{\comp(\Des_{\dI}(\stdz(T)))}(X,Y)$.  

Now suppose $x_{a_1}\cdots x_{a_k}y_{b_1}\cdots y_{b_m}$ is a monomial in $\tilde{Q}_{\comp(\Des_{\dI}(S))}(X,Y)$ for some standard immaculate tableau $S$ of shape $\alpha$.  We must show that there exists a hook immaculate tableau with content $a_1,\ldots,a_k$, $b_1',\ldots,b_m'$.  We do this by replacing $n$ in $S$ with $b_m'$, $n-1$ in $S$ with $b_{m-1}'$ and so on.  Since $b_i=b_{i+1}$ implies that $i+k \in \Des_{\dI}(S)$, we have that each primed entry in a row is distinct and increasing from left to right.  Similarly, if $a_i=a_{i+1}$, then $i \notin \Des_{\dI}(S)$, guaranteeing that the first column \svw{is increasing bottom to top and} has distinct unprimed entries.  Thus the result is a hook immaculate tableau of content $x_{a_1}\cdots x_{a_k}y_{b_1}\cdots y_{b_m}$. \end{proof}

Berele and Regev~\cite{BR1987} defined hook Schur functions indexed by a partition $\lambda$ as 
\[\mathcal{H}s_\lambda(X,Y)=\sum_{\mu\subseteq \lambda} s_\mu(X)s_{\lambda'/\mu'}(Y).\]
We have the following analogue of Theorem~\ref{thm:psiresults}, \svw{Point} \eqref{schurtodI}.
\begin{theorem}\label{thm:hookexpansion}
Let $\lambda$ be a partition.  Then 
\[\mathcal{H}s_\lambda(X,Y) = \sum_{\tau \in S_{\ell(\lambda)}} (-1)^{\sgn(\tau)} \hdI_{\tau(\lambda)}(X,Y).\]
\end{theorem}

\begin{proof}
Let $\lambda$ be a partition.  Then by \svwMar{Theorem~\ref{thm:psiresults}, {Point} \eqref{schurtodI} and Corollary~\ref{cor:schurexprsdI}},
\begin{align}
\mathcal{H}s_\lambda(X,Y) &= \sum_{\mu \subseteq \lambda} s_\mu(X) s_{\lambda'/\mu'}(Y)\nonumber\\
&=\sum_{\mu\subseteq\lambda}\left(\sum_{\sigma \in S_{\ell(\mu)}}(-1)^{\sgn(\sigma)}\dI_{\sigma(\mu)}(X)s_{\lambda'/\mu'}(Y)\right)\nonumber\\
&= \sum_{\mu\subseteq\lambda}\left(\sum_{\sigma \in S_{\ell(\mu)}}\dI_{\sigma(\mu)}(X)\sum_{\tau \in S_{\ell(\lambda)}}(-1)^{\sgn(\tau)}\rdI_{\tau(\lambda)/\sigma(\mu)}(Y)\right)\nonumber\\
&=\sum_{\tau \in S_{\ell(\lambda)}} (-1)^{\sgn(\tau)} \left(\sum_{\mu \subseteq\lambda} \sum_{\sigma \in S_{\ell(\mu)}} \dI_{\sigma(\mu)}(X)\rdI_{\tau(\lambda)/\sigma(\mu)}(Y)\right).\label{eqn:hookschurproof}
\end{align}
Note that the only terms $\sigma(\mu)$ that appear in~\eqref{eqn:hookschurproof} are those such that $\sigma(\mu)=\beta$ for a composition $\beta$.  We rewrite~\eqref{eqn:hookschurproof} as 
\begin{align*}
\mathcal{H}s_\lambda(X,Y) &=\sum_{\tau \in S_{\ell(\lambda)}} (-1)^{\sgn(\tau)} \left(\sum_{\mu \subseteq\lambda} \sum_{\sigma \in S_{\ell(\mu)}} \dI_{\sigma(\mu)}(X)\rdI_{\tau(\lambda)/\sigma(\mu)}(Y)\right)\\
&=\sum_{\tau \in S_{\ell(\lambda)}}(-1)^{\sgn(\tau)}\sum_{\svwMar{\beta \subseteq \tau(\lambda)}}\dI_\beta(X) \rdI_{\tau(\lambda)/\beta}(Y)\\
&=\sum_{\tau \in S_{\ell(\lambda)}}(-1)^{\sgn(\tau)} \hdI_{\tau(\lambda)}(X,Y). \qedhere
\end{align*}
\end{proof}

\bibliographystyle{plain}
\bibliography{2023April2AIAMFINAL}

\end{document}